\documentclass[11pt,reqno]{amsart} 
\usepackage[margin=1in]{geometry}
\usepackage{appendix}
\usepackage{stackrel}
\usepackage{bm}
\usepackage{amsfonts}
\usepackage{setspace}
\usepackage{amsmath}
\usepackage{bbm}
\usepackage{nicefrac}
\usepackage{amsmath,amsthm,amssymb}
\usepackage{graphicx}
\usepackage{xcolor}
\usepackage{dsfont}
\usepackage{amsmath}
\usepackage{amssymb}
\usepackage{mathtools}

\newtheorem{theorem}{Theorem}[section]
\newtheorem{lemma}[theorem]{Lemma}
\newtheorem{proposition}[theorem]{Proposition}

\newtheorem{corollary}[theorem]{Corollary}
\newtheorem{definition}[theorem]{Definition}
\newtheorem{remark}[theorem]{Remark}
\newtheorem{Thm}{Theorem}

\newtheorem{conjeture}[Thm]{Conjeture}

\DeclareMathOperator*{\dist}{dist}
\newcommand{\N}{\mathbb{N}}

\newcommand{\Z}{\mathbb{Z}}

\newcommand{\R}{\mathbb{R}}
\newcommand{\C}{\mathbb{C}}
\newcommand{\E}[1]{\mathbb{E}\left[#1\right]}
\newcommand{\rL}{\right\}}
\newcommand{\lL}{\left\{}
\newcommand{\rC}{\right ]}
\newcommand{\lC}{\left [}
\newcommand{\rP}{\right)}
\newcommand{\lP}{\left(}
\newcommand{\rabs}{\right|}
\newcommand{\labs}{\left|}
\newcommand{\abs}[1]{\left| #1 \right|}
\newcommand{\absD}[1]{\left|\left| #1 \right|\right|}
\newcommand{\Prob}[1]{\mathbb{P}\left(#1\right)}
\usepackage{color}

\usepackage{lineno}


\usepackage{blindtext}

\title[Random Toeplitz Matrices]{{Random Toeplitz Matrices: The Condition Number under High Stochastic Dependence}} 
\author{Manrique--Mir\'on, Paulo}
\address{}
\email{pmanriquem@ipn.mx}
\thanks{}
\date{\today}
\begin{document}
\maketitle
\begin{abstract}
In this paper, we study the condition number of a random Toeplitz matrix. Since a Toeplitz matrix is a diagonal constant matrix, its rows or columns cannot be stochastically independent. This situation does not permit us to use the classic strategy to analyze its minimum singular value when all the entries of a random matrix are stochastically independent. Using a circulant embedding as a decoupling technique, we break the stochastic dependence of the structure of the Toeplitz matrix and reduce the problem to analyze the extreme singular values of a random circulant matrix. A circulant matrix is, in fact, a particular case of a Toeplitz matrix, but with a more specific structure, where it is possible to obtain explicit formulas for its eigenvalues and also for its singular values. Among our results, we show the condition number of a non--symmetric random circulant matrix $\mathcal{C}_n$ of dimension $n$ under the existence of moment generating function of the random entries is $\kappa\lP\mathcal{C}_n\rP =\mbox{O}\lP\frac{1}{\varepsilon} n^{\rho+1/2} \lP\log n\rP^{1/2}\rP$ with probability $1-\mbox{O}\lP \lP\varepsilon^2 + \varepsilon\rP n^{-2\rho} + n^{-1/2+\textnormal{o}(1)}\rP$ for any $\varepsilon >0$, $\rho\in(0,1/4)$. Moreover, if the random entries only have the second moment, the condition number satisfies $\kappa\lP\mathcal{C}_n\rP = \mbox{O}\lP \frac{1}{\varepsilon} n^{\rho+1/2} \log n\rP$  with probability $1-\mbox{O}\lP\lP\varepsilon^2 + \varepsilon\rP n^{-2\rho} + \lP\log n\rP^{-1/2}\rP$. Also, we analyze the condition number of a random circulant symmetric  matrix $\mathcal{C}^{sym}_n$. For the condition number of a random (non--symmetric or symmetric) Toeplitz matrix $\mathcal{T}_n$ we establish $\kappa\lP\mathcal{T}_n\rP \leq \kappa\lP\mathcal{C}_{2n}\rP \lP\sigma_{\min}\lP C_{2n} \rP\sigma_{\min}\lP S_n \rP\rP^{-1}$, where $\sigma_{\min}(A)$ is the minimum singular value of the matrix $A$. The matrix $C_{2n}$ is a random circulant matrix and $S_n:=F^*_{2,n} D_{1,n}^{-1}F_{2,n} + F^*_{4,n} D^{-1}_2 F_{4,n}$, where $F_{2,n},F_{4,n}$ are deterministic matrices, $F^*$ indicates the conjugate transpose of $F$, and $D_{1,n}, D_{2,n}$ are random diagonal matrices. From random experiments, we conjeture that $S_n$ is well conditioned if the moment generating function of the random entries of $\mathcal{C}_{2n}$ exists.
\end{abstract}

\section{Introduction}
The singularity of random matrices has been an intensely studied topic in the last years; see e.g., \cite{BorCha2012, Liv, Luh2018,RV1,RV}. Recall that a square matrix is called singular if its determinant is zero. A criterium to determine a matrix is singular is to verify if its minimum singular value is zero.  The singular values of a matrix carry more useful information about the properties of the matrix, inclusive if it is rectangular. For example, they play an important aspect in the celebrated Circular Law Theorem; see \cite{BorCha2012} for a systematic presentation. The singular values of a (square or rectangular) matrix $A$ are the eigenvalues of the matrix $\sqrt{A^T A}$, where $A^T$ denotes the transpose matrix of $A$. 

The extreme singular values are related to the operator norm of a matrix. The operator norm of an $n$-dimensional square matrix $A$ is defined by  
\[ 
\|A\| :=\max_{\|x\|_2 = 1} \| Ax \|_2, 
\]
where $\|\cdot\|_2$ denotes the Euclidean norm. If $0\leq\sigma_n\leq \sigma_{n-1}\leq\cdots\leq\sigma_1$ are the singular values of matrix $A$, we have,
\[
\|A\| =\max_{\|x\|_2 = 1} \| Ax \|_2 = \sigma_1, \;\;\;\;\; \left\|A^{-1}\right\| =\lC\min_{\|x\|_2 = 1} \| Ax \|_2 \rC^{-1}= \sigma_n^{-1}.
\] 
The last equality has only meaning when $A$ is non-singular. In the rest of this paper, we consider the following notation for the extreme singular values: $\sigma_{\max}:=\sigma_1$ and $\sigma_{\min}:=\sigma_n$.  In this context, it is known that $\sigma_{\min}$ measures the distance of a matrix $A$ to the set of singular matrices. More precisely, 
\[
\sigma_{\min} = \inf\lL \|E\| : A + E \mbox{ is singular and $E$ is $n\times n$ matrix} \rL.
\]
From the above identity, we can verify that if the minimum singular value is zero, then the matrix is singular. If $\sigma_{\min}\neq 0$, we can define the so--called {\it condition number} $\kappa(A)$ of a matrix $A$ as 
\[
\kappa(A):=\frac{\sigma_{\max}(A)}{ \sigma_{\min}(A)}.
\]
The condition number was independently introduced by Alan Turing (1948), and by John von Neumann and Herman Goldstine (1947) in order to study the accuracy in the solution of a linear system in the presence of finite--precision arithmetic \cite{cucker2016probabilistic}. By the definition of condition number it is easy to see $\kappa \geq 1$. If $\kappa$ is {\it very large}, the corresponding matrix is said to be {\it ill-conditioned}. The logic for this terminology is that if $\kappa$ is very large, then $\sigma_{\min}$ should be small and the matrix $A$ is close to the set of singular matrices. Then, a small perturbation of $A$ can cause loss accuracy in the computed solution of the system $Ax=b$; see \cite{demmel1987geometry}. Thus, it is interesting to set up conditions under which $\kappa$ is close to low values and this requires the estimation from below of the minimum singular value $\sigma_{\min}$ as well as the estimation from above of the maximum singular value $\sigma_{\max}$. These are precisely the main goals of this paper for the specific class of structured random matrices which are (non--symmetric and symmetric) Toeplitz matrices.\\

Among the first papers on the condition number of random matrices, we have one from Demmel \cite{demmel1987geometry}. He assumes that $A$ is an $n$-dimensional random square matrix such that $A/\|A\|_F$ ($\|\cdot\|_F$ is the Frobenius norm) is uniformly distributed on the unit sphere. Demmel defines $\kappa_1(A) := \| A\|_F \left\| A^{-1} \right\|$ as an approximation to the condition number and shows
\[
\frac{C (1-1/x)^{n^2-1}}{x} \leq \Prob{\kappa_1(A)\geq x} \leq \sum_{k=1}^{n^2} 2\binom{n^2}{k} \lP \frac{2n}{x}\rP^k,
\] where $C>0$ depends on $n$.\\

Other papers study the behavior of $\sigma_{\min}$ or the condition number of random matrices under either strong independency assumptions or some structure specification on their entries. For example, Rudelson and Vershynin \cite{RV1} prove that if $A$ has entries which are independent and identically distributed (i.i.d. for short) in the class of sub--Gaussian random variables (r.v. for short) with variance at least 1, then for all $\varepsilon\geq 0$, $\Prob{\sigma_{\min}(A) \leq \varepsilon n^{-1/2}} \leq C\varepsilon + c^n$, for some constants $C>0$ and $c\in(0,1)$ depending on the sub--Gaussian r.v. Vershynin \cite{vershynin2014invertibility} proves a similar estimation for a symmetric matrix where the upper triangle part has independent and identically sub--Gaussian r.v. entries. Recently, Litvak and et al. \cite{Lit} consider that $A$ is a matrix with sub--Gaussian i.i.d. entries with zero mean and unit variance. They show $\Prob{\kappa(A) \leq n/t}\leq 2\exp(-ct^2)$ for $t\geq 1$ and positive constant $c$ which depends on the sub--Gaussian r.v.\\

On the other hand, random matrices with structure have been analyzed, i.e., matrices whose entries follow certain disposition. For example, random triangular matrices $L_n$ with entries in the diagonal and below independently distributed and drawn from a standard Gaussian are poorly conditioned. In fact, Viswanath and Trefethen \cite{cucker2016probabilistic} show $\sqrt[n]{\kappa\lP L_n \rP}\to2$ almost surely as $n\to\infty$. Other kinds of structured matrices that have been analyzed are the Toeplitz matrices or the circulant matrices, which are very common objects in different areas of mathematics \cite{arup2018, David2012, gray2006toeplitz, SenVir2013}. Toeplitz matrices, for example, are used in different numerical algorithms that involve differential equations, integral equations, time series analysis, Markov chain, Fast Fourier Transform, among others \cite{ng2004iterative, pan2001structured}.\\

In the circulant case, Meckes \cite{Mec2009} proves that a random circulant matrix with Bernoulli entries is non--singular with probability going to $1$ when its dimension goes to $\infty$. Bose et al. \cite{AruRajKou2009} study the behavior of $\sigma_{\max}$ for random circulant-type matrices. The convergence of $\sigma_{\max}$ for random symmetric Toeplitz matrices is studied by Sen and Vir\'ag \cite{SenVir2013}, Bose and Sen \cite{aruSen2007}, while Adamczak \cite{Ada} gives bounds for $\sigma_{\max}$ of random rectangular Toeplitz matrices. Meckes \cite{Mec2007} shows $\sigma_{\max}=\textnormal{O}\lP \sqrt{n\log n}\rP$ for random (non--symmetric or symmetric) Toeplitz matrices with sub--Gaussian entries. Bose and Saha \cite{arup2018} recollect many results about circulant-type random matrices, for instance the limit of its {\it empirical spectral distribution function} and the convergence of $\sigma_{\max}$ with appropriate normalization. Pan, Svadlenka, and Zhao \cite{panSZ2015} estimate the condition number for (non--symmetric) random circulant and Toeplitz matrices with i.i.d. Gaussian random entries. Their approach and results are different from those derived here. \\

Note that in some cases mentioned above, it is assumed that a random matrix has O$(n^2)$ independent r.v. In the Toeplitz matrix, we can consider at most $2n-1$ random independent entries and in the circulant case at most $n$. In the symmetric case for Toeplitz and circulant, we have at most $n$ and $n/2$ independent random entries, respectively. A random matrix with all independent entries has stochastically independent rows and columns. This permits us to follow a strategy to estimate the value of $\sigma_{\min}$ as in \cite{Lit,RV1,vershynin2014invertibility}. Meanwhile, all the rows or columns of a random Toeplitz or circulant matrix are strongly stochastically dependent. Thus, the estimation of the extreme singular values of a random Toeplitz or circulant matrix needs a priori different approach.\\

A circulant matrix is a particular case of a Toeplitz matrix and its structure permits to give explicit expressions for its eigenvalues. In fact, circulant matrices have a lot of useful properties (see \cite{pan2001structured}). It is well-known that a Toeplitz matrix can be approximated by circulant matrices (see \cite{gray2006toeplitz,SenVir2013}). Thus, our strategy will be reduced the Toeplitz problem to the study of the extreme singular values of a random circulant matrix.\\

The main tools used to bring the Toeplitz problem to the circulant problem are the Cauchy Interlacing Theorem and the {\it circulant embedding}, which can be considered as a {\it decoupling technique}. They permit to break the strong stochastic dependence in the Toeplitz structure into circulant structure, where we can handle the estimation of the extreme singular values. Once the problem is reduced to circulant structure, we estimate $\sigma_{\max}$ by estimating the maximum modulus of a random polynomial on the unit circle. To do this, we use the so--called {\it Salem--Zygmund inequality}. Under the existence of the moment generating function (m.g.f. for short), we show that $\sigma_{\max}$ of a random (non--symmetric or symmetric) circulant matrix is $\mbox{O}\lP\lP n \log n\rP^{1/2}\rP$ with probability $1-\mbox{O}\lP n^{-2}\rP$. In the non--symmetric case, we can relax our conditions up to the existence of only the second moment, in which case $\sigma_{\max}$ is $\mbox{O}\lP n^{1/2} \log n\rP$ with probability $1-\mbox{O}\lP \lP \log n\rP^{-1/2}\rP$. On the other hand, using the concept of {\it least common denominator}, a tool developed to handle the so--called {\it small ball probability problem}, we obtain a lower bound for $\sigma_{\min}$. For any $\varepsilon>0$ and $\rho\in(0,1/4)$, we show that under mild conditions (see below the condition \eqref{fanto1}) $\sigma_{\min}$ of a random non--symmetric circulant matrix is at least $\varepsilon n^{-\rho}$ with probability $\mbox{O}\lP \frac{\varepsilon^2+\varepsilon}{n^{2\rho}} + \frac{1}{n^{1/2-\textnormal{o}(1)}}\rP$. In the symmetric case, we show that $\sigma_{\min}$ is at least $\varepsilon n^{-0.51}$ with probability $\mbox{O}\lP \frac{\varepsilon}{n^{0.1}} + \frac{1}{n^{77/300-\textnormal{o}(1)}}\rP$ for any $\varepsilon>0$. 
{
From our results in random circulant matrices, we can estimate with high probability the condition number of some random rectangular Toeplitz matrices (Theorem \ref{cor250820200741}). Also, we give an upper bound for $\sigma_{\max}$ of a random square Toeplitz matrix (Theorems \ref{thm:TnSM}, \ref{thm:TnSM1}) and we conjecture it is well conditioned when its entries are r.v. with m.g.f. (Conjeture \ref{con010920201139}). Since a Hankel matrix can be transformed into a Toeplitz matrix as we will see later, all our results for random Toeplitz matrices hold for random Hankel matrices.}\\

This paper is organized as follows. In Section \ref{sec:MainResults} we state the main results of this paper. The reduction of the Toeplitz problem to the circulant problem is explained in Section \ref{sec:ToeplitzToCirculant}. The results on random circulant matrices relating to random polynomial are stated in Section \ref{sec:Circulant}. In Section \ref{sec:salemZygmund} we prove the Salem--Zygmund inequalities for non--symmetric and symmetric cases. In Section \ref{sec:LowestSingV1} we establish a lower bound for $\sigma_{\min}$ of a random non--symmetric circulant matrix. Meanwhile, in Section \ref{sec:minSinSymCir} we give a lower bound for $\sigma_{\min}$ of a random symmetric circulant matrix. In Section \ref{sec:conNumToe} we give the proof of our results on the condition number of a random circulant matrix. The Appendix \ref{app10042020} and Appendix \ref{app30042020} contain additional material in order to provide clarity to this paper. \\

{\bf Acknowledgments.} I would like to thank Jes\'us L\'opez Estrada for his suggestions to improve the presentation of this work. 

\section{Main results}\label{sec:MainResults}

A Toeplitz matrix $\mathcal{T}_n$ is an $n\times n$ matrix with constant diagonals, i.e., $\mathcal{T}_n$ has the following structure
\[
\mathcal{T}_n =
\left[ \begin{array}{ccccc}
\xi_0 & \xi_1  & \ldots & \xi_{n-2}  &Ê\xi_{n-1} \\
\xi_{-1} & \xi_0  & \xi_{1} &   &Ê\xi_{n-1} \\
\vdots & \xi_{-1}  & \xi_0 & \ddots  &Ê\vdots \\
\xi_{-n+2} & & \ddots & \ddots & \xi_{1} \\
\xi_{-n+1} & \xi_{-n+2} & \cdots & \xi_{-1} & \xi_0
\end{array} \right].
\] 

When the entries of $\mathcal{T}_n$ are r.v., we say it is a random Toeplitz matrix. Let $\Xi:=\lL \xi_j : j\in\Z \rL$ be a set of i.i.d. r.v. We assume that the random entries of $\mathcal{T}_n$ belong to $\Xi$.\\

In the rest of this paper, any positive constant will be denoted by $C,C_0,C_1,C_2,\ldots$, which are not necessarily equal in each statement as they appear. We denote the norm of a real or complex number $z$ as $\abs{z}$.\\

First, we show that Toeplitz matrices and Hankel matrices are equivalent. Later, we state our results on Toeplitz and circulant matrices.

\subsection{Random Hankel matrix} Let $J:=(J_{i,j})$ be the $n\times n$ {\it exchange matrix}, i.e., the entries of $J$ are $J_{i,j}=1$ if $j=n-i+1$, and $J_{i,j}=0$ if $j\neq n-i+1$. Note that $J$ has the following properties: 
\begin{itemize}
\item $J^2=I_n$, where $I_n$ is the $n\times n$ identity matrix
\item $J^T = J$
\end{itemize}

An $n\times n$ matrix $H$ is called Hankel if $JH$ is a Toeplitz matrix. For more details on Hankel matrices see \cite{pan2001structured}. Observe\begin{equation}\label{eqn17052020}\sqrt{\lP JH \rP^T\lP J H\rP} = \sqrt{H^T J^T J H} = \sqrt{H^T H}.\end{equation} By \eqref{eqn17052020} we have that $JH$ and $H$ have the same singular values. Then, all results in this section hold for random (non--symmetric or symmetric) Hankel matrices under the corresponding assumptions.

\subsection{Non--symmetric Toeplitz} Let $\xi_j\in\Xi$ for $j=-n,\ldots,n$, and we consider the respective random Toeplitz matrix $ \mathcal{T}_n$ with $2n-1$ i.i.d. entries. For the first result on the Toeplitz matrix, we assume that the random entries have m.g.f. The existence of m.g.f. permits to use {\it Chernoff bounding technique} to estimate $\sigma_{\max}$.

\begin{theorem}[Non--symmetric Toeplitz: Maximum singular value I]\label{thm:TnSM} Suppose $\xi_0$ has zero mean and finite positive variance. If the m.g.f. of $\xi_0$ exists in an open interval {around zero, then}
\[
\Prob{\sigma_{\max}\lP \mathcal{T}_n\rP \geq C_0 \lP (2n)\log(2n) \rP^{\nicefrac{1}{2}}} \leq \frac{C_1}{(2n)^2},
\] where $C_0, C_1$ depend on the distribution of $\xi_0$.
\end{theorem}

Actually, we can relax the conditions in Theorem \ref{thm:TnSM} up to the existence of the second moment of $\xi_0$.

\begin{theorem}[Non--symmetric Toeplitz: Maximum singular value II]\label{thm:TnSM1} Suppose $\xi_0$ has zero mean and $\E{\xi_0^2}<\infty$ exists. Then,
\[
\Prob{\sigma_{\max}\lP \mathcal{T}_n\rP \geq C_0 \lP 2n\rP^{\nicefrac{1}{2}} \log(2n)} \leq \frac{C_1}{\lP\log(2n)\rP^{1/2}},
\] where $C_0, C_1$ depend on the distribution of $\xi_0$.
\end{theorem}

{
\begin{remark}\label{rem140720200654} If the entries of a non--symmetric Toeplitz are independent, uniformly sub--Gaussian r.v. with zero mean but no necessarily identically distributed, Meckes \cite{Mec2007} shows $\E{\sigma_{\max}\lP \mathcal{T}_n\rP}\leq C\sqrt{n\log n}$ for a constant $C$ depending on sub--Gaussian r.v. Additionally, he establishes that if the random entries are almost surely bounded or satisfy the so--called logarithmic Sobolev inequality, we have $\limsup_{n\to\infty}\frac{\sigma_{\max}\lP \mathcal{T}_n\rP}{\sqrt{n}}\leq C_1$ for some constant $C_1$. Here, under the existence of m.g.f, which is a weaker condition than uniformly sub--Gaussian, Theorem \ref{thm:TnSM} shows how decrease the tail of $\sigma_{\max}\lP \mathcal{T}_n\rP$. From this, we obtain directly by the Borel--Cantelli Lemma that,
\[
\limsup_{n\to\infty} \frac{\sigma_{\max}\lP \mathcal{T}_n\rP}{\lP (2n)\log (2n)\rP^{1/2}} \leq C_0 \;\;\;\mbox{almost surely}.
\]
Meckes \cite{Mec2007} shows under the existence of the second moment of the random entries that $\E{\sigma_{\max}\lP \mathcal{T}_n\rP}\geq C_2\sqrt{n\log n}$. As we will see later, we also have that $\E{\sigma_{\max}\lP \mathcal{T}_n\rP}\leq C_3\sqrt{n\log n}$, meaning, \[C_2\sqrt{n\log n}\leq  \E{\sigma_{\max}\lP \mathcal{T}_n\rP} \leq C_3 \sqrt{n\log n}.\] The results on $\sigma_{\max}\lP\mathcal{T}_n\rP$ in this paper are direct consequence of our statements on random polynomials, which is a different approach than \cite{Mec2007}. This is due to the circulant embedding permits reduce the Toeplitz problem to the circulant problem.
\end{remark}
}
\subsection{Symmetric Toeplitz} Let $\xi_j\in\Xi$ for $j=0,1,\ldots,n$ and we consider the corresponding random symmetric Toeplitz matrix $\mathcal{T}^{sym}_n$ with $n$ i.i.d. entries. In the following, we give an estimation of $\sigma_{\max}\lP\mathcal{T}^{sym}_n\rP$.

\begin{theorem}[Symmetric Toeplitz: Maximum singular value]\label{thm:TSM} Suppose $\xi_0$ has zero mean and finite positive variance. If the m.g.f. of $\xi_0$ exists in an open interval {around zero, then}
\[
\Prob{\sigma_{\max}\lP \mathcal{T}^{sym}_n\rP \geq C_0 \lP (2n)\log(2n) \rP^{\nicefrac{1}{2}}} \leq \frac{C_1}{(2n)^2},
\] where the $C_0, C_1$ depend on the distribution of $\xi_0$.
\end{theorem}

{
\begin{remark} As it was mentioned in Remark \ref{rem140720200654}, Meckes \cite{Mec2007} studies  $\E{\sigma_{\max}\lP \mathcal{T}_n\rP}$ for non--symmetric random Toeplitz matrix, but actually he establishes his results for symmetric random Toeplitz matrices. Thus, the same comments in Remark \ref{rem140720200654} can be applied to the symmetric case. From Theorem \ref{thm:TSM}, we obtain by the Borel--Cantelli Lemma that $\limsup_{n\to\infty} \frac{\sigma_{\max}\lP \mathcal{T}^{sym}_n\rP}{\lP (2n)\log (2n)\rP^{1/2}} \leq C_0$ almost surely.
\end{remark}
}

\subsection{Singularity of a Toeplitz matrix} As we will see in Section \ref{sec:ToeplitzToCirculant}, it is possible to establish upper and lower bounds for all singular values of $\mathcal{T}_n$, no matter it is non--symmetric or symmetric matrix. More precisely, for all $i$ we have,
\[
\sigma^{2}_{\min}\lP \mathcal{C}_{2n} \rP \sigma_{i}\lP \mathcal{S}_n\rP  \leq \sigma_i\lP \mathcal{T}_n\rP \leq \sigma^{2}_{\max}\lP \mathcal{C}_{2n}\rP \sigma_{i}\lP \mathcal{S}_n\rP,
\] where $\mathcal{C}_{2n}$ is a circulant matrix of size $2n$ and $\mathcal{S} _n=F^*_{2,n} D_{1,n}^{-1}F_{2,n} + F^*_{4,n} D^{-1}_{2,n} F_{4,n}$, where $F_{2,n}, F_{4,n}$ are deterministic matrices, $F^*_{2,n}$ indicates the conjugate transpose of $F_{2,n}$, and $D_{1,n}, D_{2,n}$ are random diagonal matrices. From our results in circulant matrix, we see that $\sigma^{2}_{\min}\lP \mathcal{C}_{2n}\rP, \sigma^{2}_{\max}\lP \mathcal{C}_{2n}\rP$ have a nice behavior and they assure that $D_{1,n}$ and $D_{2,n}$ are invertible with high probability. Moreover, we have, \begin{equation}\label{eqn040920201642}\kappa\lP\mathcal{T}_n\rP \leq \frac{\kappa\lP\mathcal{C}_{2n}\rP}{\sigma_{\min}\lP \mathcal{C}_{2n} \rP\sigma_{\min}\lP S_n \rP}.\end{equation}

 The random part of $\mathcal{S}_n$ is composing by diagonal matrices whose diagonal entries are $G^{-1}_n(w_{2n}^k)$, $k=0,\ldots,2n-1$, where $G_{2n}(z):=\sum_{j=0}^{2n-1} \xi_j z^j$ is a random polynomial and $w_{2n}=\exp\lP i 2\pi\frac{1}{2n}\rP$. We conjeture $\mathcal{S}_n$ is well conditioned.

\begin{conjeture}[Minimum singular value]	\label{con010920201139} If the coefficients of $G_{2n}(z):=\sum_{j=0}^{2n-1} \xi_j z^j$ are r.v. with m.g.f., then $\mathcal{S}_n$ is well conditioned.
\end{conjeture}

If Conjeture \ref{con010920201139} is true, then (non--symmetric or symmetric) $\mathcal{T}_n$ will be well conditioned. 

\subsection{Condition number of a random circulant matrix} To analyze $\kappa\lP\mathcal{C}_{n}\rP$ we need to give appropriate bounds for $\sigma_{\max}\lP\mathcal{C}_{n}\rP$ and $\sigma_{\min}\lP\mathcal{C}_{n}\rP$. Theorems \ref{thm:TnSM}, \ref{thm:TnSM1}, \ref{thm:TSM} are direct consequence of our results on $\sigma_{\max}$ for a random circulant matrix $\mathcal{C}_{n}$ of size $n$. Thus, the bounds in these statements hold for $\sigma_{\max}$ of the circulant case. For $\sigma_{\min}$, we have the following results. But before, we need to introduce a condition. We say a r.v. $\xi$ satisfies the condition \eqref{fanto1} if  
\begin{equation}\label{fanto1}
\tag{H}
\sup_{u\in\mathbb{R}} \Prob{ |\xi - u| \leq 1} \leq 1-q\;\; \mbox{ and } \;\;\Prob{ |\xi|>M} \leq q/2 
\end{equation}
for some  $M>0$ and $q\in (0,1)$.\\

The first part of Condition \eqref{fanto1} says that the r.v. $\xi$ is not {\it concentrated} around any single value. This is related to the {\it L\'evy concentration function}, which in general is defined as follows.

\begin{definition}\label{def06012020}
The L\'evy concentration function of a random vector $\xi\in\R^n$ is defined for any $\varepsilon \geq 0$ as 
\[
\mathcal{L}\lP\xi, \varepsilon\rP := \sup_{x\in\R^n} \Prob{\|{\xi - x}\|_2\leq \varepsilon}.
\]
\end{definition}

{
\begin{remark} For a fixed $\varepsilon>0$, we have $\mathcal{L}\lP\xi, \varepsilon\rP=\mathcal{L}\lP\varepsilon^{-1}\xi,1\rP$. On the other hand, if $\sigma(A)$ is a singular value of the matrix $A$, then for a fixed $w\in\R$, we have $\sigma(wA)=\abs{w}\sigma(A)$. Hence, all our results can be applied to all bounded r.v., using a suitable scaling, as the continuous uniform distribution on $(0,1)$ or the Rademacher distribution (uniform on $\lL-1,1\rL$). In fact, all r.v. with second moment satisfy Condition \eqref{fanto1}.
\end{remark}
}

\begin{theorem}[Non--symmetric circulant: Minimum singular value]\label{thm01042018}
Suppose $\xi_0$ satisfies condition \eqref{fanto1}. Then, for any $\varepsilon > 0$ and $\rho\in(0,1/4)$, we have for all large $n$,
\begin{equation*}\label{upt2}
\Prob{\sigma_{\min}(\mathcal{C}_n)\leq  \varepsilon n^{-\rho}} \leq C\lP \frac{\varepsilon^2+\varepsilon}{n^{2\rho}} + \frac{1}{n^{1/2-\textnormal{o}(1)}}\rP, 
\end{equation*} where $C$ depends on the distribution of $\xi_0$.
\end{theorem}

The next result show how small can be $\sigma_{\min}\lP\mathcal{C}_n^{sym}\rP$ for a symmetric random circulant matrix $\mathcal{C}_n^{sym}$. 

\begin{theorem}[Symmetric circulant: Minimum singular value]\label{thm:LowestSingV} Suppose $\xi_0$ satisfies condition \eqref{fanto1}. Then for any $\varepsilon >0$ and all large $n$,
\[
\Prob{\sigma_{\min}\lP\mathcal{C}_n^{sym}\rP  \leq \varepsilon n^{-0.51}}  \leq C\lP \frac{\varepsilon}{n^{0.1}} + \frac{1}{n^{77/300-\textnormal{o}(1)}}\rP,
\] where $C$ depends on the distribution of $\xi_0$.
\end{theorem}

If the conditions of the previous results on $\sigma_{\max}\lP \mathcal{C}_n\rP$ and $\sigma_{\min}\lP \mathcal{C}_n\rP$ hold at the same time, we can bound the condition number of a random circulant matrix $\mathcal{C}_n$.

\begin{theorem}[Non--symmetric circulant: Condition number]\label{thm10042020_1} Suppose the conditions of Theorem \ref{thm:TnSM} and Theorem \ref{thm01042018} hold. Then, for any $\varepsilon>0$ and $\rho\in(0,1/4)$, the condition number $\kappa\lP\mathcal{C}_n\rP$ of a random (non--symmetric) circulant matrix $\mathcal{C}_n$ satisfies for all large $n$,
\[
 \Prob{\kappa\lP\mathcal{C}_n\rP \leq \frac{C_0}{\varepsilon} n^{\rho+1/2} \lP\log n\rP^{1/2}} \geq 1-C_1\lP \lP\varepsilon^2 + \varepsilon\rP n^{-2\rho} + n^{-1/2+\textnormal{o}(1)}\rP,
\] where $C_0, C_1$ depend on the distribution of $\xi_0$. If the conditions of Theorem \ref{thm:TnSM1} and Theorem
\ref{thm01042018} hold, the condition number satisfies for all large $n$,
\[
\Prob{\kappa\lP\mathcal{C}_n\rP \leq \frac{C_0}{\varepsilon} n^{\rho+1/2} \log n} \geq 1-C_2\lP\lP\varepsilon^2 + \varepsilon\rP n^{-2\rho} + \lP\log n\rP^{-1/2}\rP,
\] where $C_0, C_2$ depend on the distribution of $\xi_0$.
\end{theorem}

\begin{theorem}[Symmetric circulant: Condition number]\label{thm10042020_2} If the conditions of Theorem \ref{thm:TSM} and Theorem
\ref{thm:LowestSingV} hold. Then, for any $\varepsilon>0$ the condition number $\kappa\lP \mathcal{C}^{sym}_n\rP$ satisfies for all large $n$,
\[
\Prob{\kappa\lP \mathcal{C}^{sym}_n\rP \leq \frac{C_0}{\varepsilon} n^{1.01} \lP\log n\rP^{1/2}} \geq 1-C_1\lP \varepsilon n^{-0.1} + n^{-77/300+\textnormal{o}(1)}\rP,
\] where $C_0, C_1$ depend on the distribution of $\xi_0$. 
\end{theorem}
 
\section{From Toeplitz matrices to circulant matrices}\label{sec:ToeplitzToCirculant}

A circulant matrix $\mathcal{C}_n$ is a particular case of a Toeplitz matrix of dimension $n$ where the entries are circulated row by row. A circulant matrix $\mathcal{C}_n$ has the following form
\[
\mathcal{C}_n =
\left[ \begin{array}{ccccc}
\xi_0 & \xi_1 & \cdots & \xi_{n-2} & \xi_{n-1} \\
\xi_{n-1} & \xi_0 & \cdots & \xi_{n-3} & \xi_{n-2} \\
\xi_{n-2} & \xi_{n-1} & \cdots & \xi_{n-4} & \xi_{n-3} \\
\vdots & \vdots & \ddots & \vdots & \vdots \\
\xi_1 & \xi_2 & \cdots & \xi_{n-1} & \xi_0
\end{array} \right].
\] 

Note that a circulant matrix is defined by its first row. Let $w_n:= \exp\lP i\frac{2\pi}{n}\rP$, $i^2=-1$. It is well known that any circulant matrix is diagonalized by the Fourier matrix $F_n$, whose entries are powers of $w_n$, i.e., $F_n=\frac{1}{\sqrt{n}} \lP w_n^{jk}\rP_{0\leq j,k\leq n-1}$. By a straightforward computation, it follows:
\begin{equation}\label{eqn12092020}
\mathcal{C}_n = F_n^{*} \mbox{diag}\lP G_n(1),G_n(w_n),\ldots,G_n(w_n^{n-1}) \rP F_n, 
\end{equation} where $F_n^{*}$ is the conjugate transpose of $F_n$, and $G_n(z):=\sum_{j=0}^{n-1} \xi_j z^j$ is a complex polynomial. This property permits to find explicit expressions for the eigenvalues of a circulant matrix. If $\lambda_0,\ldots,\lambda_{n-1}$ denote the eigenvalues, then we have,
\begin{equation}\label{eqn09042020}
\lambda_k = G_n(w_n^k)= \sum_{j=0}^{n-1} \xi_j w_n^{jk}\;\;\mbox{for $k=0,\ldots, n-1$}. 
\end{equation}

If the circulant matrix is symmetric, the expressions for eigenvalues are reduced to a linear combination of cosine values, i.e., they can be expressed as:
\begin{itemize}
\item $n$ odd: $\lambda_k=\lambda_{n-k}$ for $1\leq k\leq \lfloor n/2\rfloor$, then
\begin{equation}\label{eqn:1453}
\lambda_0 =  \xi_0 + 2\sum_{j=1}^{\lfloor n/2\rfloor} \xi_j, \;\;\;
\lambda_k =  \xi_0 + 2\sum_{j=1}^{\lfloor n/2\rfloor} \xi_j \cos\lP\frac{2\pi k}{n} j\rP,
\end{equation}

\item $n$ even: $\lambda_k=\lambda_{n-k}$ for $1\leq k\leq n/2$, then
\begin{equation}\label{eqn:1454}
\lambda_0 = \xi_0 + 2\sum_{j=1}^{n/2 -1} \xi_j + \xi_{n/2}, \;\;\;
\lambda_k = \xi_0 + 2\sum_{j=1}^{n/2 -1} \xi_j \cos\lP\frac{2\pi k}{n} j\rP + (-1)^k\xi_{n/2}.
\end{equation}
\end{itemize}

In fact, a circulant matrix with real or complex entries is a normal matrix, i.e., it satisfies the condition $\mathcal{C}_n^* \mathcal{C}_n = \mathcal{C}_n \mathcal{C}_n^*$, where $\mathcal{C}^*_n$ denotes the conjugate transpose of $\mathcal{C}_n$. This property implies that the extreme singular values of a circulant matrix satisfy the following relationships:
\[
\sigma_{\max}(\mathcal{C}_n) = \max_{k=0,\ldots,n-1} \abs{\lambda_k}, \;\;\;\; \sigma_{\min}(\mathcal{C}_n) = \min_{k=0,\ldots,n-1} \abs{\lambda_k}.
\]
Hence, the condition number of a circulant matrix is \[\kappa\lP \mathcal{C}_n \rP=\lP \max_{k=0,\ldots,n-1} \abs{\lambda_k} \rP \lP \min_{k=0,\ldots,n-1} \abs{\lambda_k}\rP^{-1}.\] 

\subsection{Circulant embedding} 

Every Toeplitz matrix $\mathcal{T}_n$ can be embedded into a circulant matrix of dimension $2n$. In fact, let $\mathcal{C}_{2n}$ be a circulant matrix defined as 
\begin{equation}\label{embCirc}
\mathcal{C}_{2n}=\left[ 
\begin{array}{cc}
\mathcal{T}_n & \mathcal{B}_n \\
\mathcal{B}_n & \mathcal{T}_n \\
\end{array}
\right],
\end{equation} where 
\[
\mathcal{B}_n :=
\left[ \begin{array}{ccccc}
\xi_* & \xi_{-n+1}  & \ldots & \xi_{-2}  &Ê\xi_{-1} \\
\xi_{n-1} & \xi_*  & \xi_{-n+1} &   &Ê\xi_{-2} \\
\vdots & \xi_{n-1}  & \xi_* & \ddots  &Ê\vdots \\
\xi_{2} & & \ddots & \ddots & \xi_{-n+1} \\
\xi_{1} & \xi_{2} & \cdots & \xi_{n-1} & \xi_*
\end{array} \right].
\] The variable $\xi_*$ does not have any restriction. Note that $\mathcal{B}_n$ is a Toeplitz matrix. If $\mathcal{T}_n$ is symmetric, then $\mathcal{C}_{2n}$ is also symmetric. This embedding is one of the key points in the development of our arguments. To see this, first we mention the Cauchy Interlacing Theorem (see \cite[Theorem 8.6.3]{golub2012matrix}) for the eigenvalues $\lambda_n(A)\leq\lambda_{n-1}(A)\leq\cdots\leq\lambda_1(A)$ of a symmetric matrix $A\in\R^{n\times n}$.

{
\begin{theorem}[Cauchy Interlacing Theorem]\label{thm_Cauchy} Let $A=[a_1|\cdots|a_n] \in \R^{m\times n}$ be a column partitioning with $m\geq n$. If $A_r=[a_1|\cdots|a_r]$, then for $r=1,\ldots,n-1$
\[
\sigma_1(A_{r+1}) \geq \sigma_1(A_{r}) \geq \sigma_2(A_{r+1})Ê\geq \cdots \geq \sigma_r(A_{r+1}) \geq \sigma_r(A_{r}) \geq \sigma_{r+1}(A_{r+1}).
\]
\end{theorem}
From the Cauchy Interlacing Theorem and the circulant embedding of $\mathcal{T}_n$, we obtain the following relationships:
\begin{equation}\label{keyObs}
\sigma_{\max}\lP \mathcal{C}_{2n} \rP \geq \sigma_{\max}\lP\left[ 
\begin{array}{c}
\mathcal{T}_n \\
\mathcal{B}_n
\end{array}
\right]\rP
\;\;\;\mbox{ and }\;\;\;
\sigma_{n}\lP\left[ 
\begin{array}{c}
\mathcal{T}_n \\
\mathcal{B}_n
\end{array}
\right]\rP \geq \sigma_{\min}\lP \mathcal{C}_{2n} \rP.
\end{equation} 
}

{
Also, we can deduce \cite[Lemma 1]{govaerts1989singular} that $\sigma_i(\mathcal{T}_n)\leq \sigma_i\lP \mathcal{C}_{2n} \rP$ for all $i$, thus \begin{equation}\label{eqn300820200819}\sigma_{\max}\lP \mathcal{T}_n \rP \leq \sigma_{\max}\lP \mathcal{C}_{2n} \rP.\end{equation} 
Unfortunately, the Cauchy Interlacing Theorem does not permit us to give a lower bound for $\sigma_{\min}(\mathcal{T}_n)$. However, from the circulant embedding we can give a lower bound for $\sigma_{\min}\lP \mathcal{T}_n\rP$, where $ \mathcal{T}_n$ can be non--symmetric or symmetric matrix. From Proposition 1 in \cite{govaerts1989singular}, we have for all $i\geq 1$,
\begin{equation}\label{keyObs1}
\sigma^{-2}_{\max}\lP \mathcal{C}_{2n} \rP \sigma_i\lP \mathcal{T}_n\rP \leq\sigma_{i}\lP \mathcal{S}_n\rP \leq \sigma^{-2}_{\min}\lP \mathcal{C}_{2n}\rP\sigma_i\lP \mathcal{T}_n\rP,
\end{equation} where $\mathcal{S}_n$ is the block in 
\[
\mathcal{C}^{-1}_{2n}=\left[ 
\begin{array}{cc}
\mathcal{P}_n & \mathcal{Q}_n \\
\mathcal{R}_n & \mathcal{S}_n \\
\end{array}
\right].
\]
In our particular case, from \eqref{eqn12092020} we observe, 
\[
\mathcal{C}^{-1}_{2n} = F_{2n}^{*} \mbox{diag}\lP G^{-1}_{2n}(1),G^{-1}_{2n}(w_{2n}),\ldots,G^{-1}_{2n}(w_{2n}^{2n-1}) \rP F_{2n},
\]
thus 
\begin{equation}\label{eqn300820200822}
\mathcal{S}_n = F^*_{2,n} D_{1,n}^{-1}F_{2,n} + F^*_{4,n} D^{-1}_{2,n} F_{4,n},
\end{equation} where the matrices $F_{2,n},F_{4,n},D_{1,n}, D_{2,n}$ are defined as \[
F_{2n}=\left[ 
\begin{array}{cc}
F_{1,n} & F_{2,n} \\
F_{3,n} & F_{4,n} \\
\end{array}
\right],
\] and \[D_{1,n}=\mbox{diag}\lP G^{-1}_{2n}(1),G^{-1}_{2n}(w_{2n}),\ldots,G^{-1}_{2n}(w_{2n}^{n-1}) \rP,\] \[D_{2,n}=\mbox{diag}\lP G^{-1}_{2n}(w_{2n}^{n}),G^{-1}_{2n}(w_{2n}),\ldots,G^{-1}_{2n}(w_{2n}^{2n-1}) \rP.\] The invertibility of $D_{1,n}$ and $D_{2,n}$ hold with high probability as well as Theorem \ref{thm01042018} and Theorem \ref{thm:LowestSingV} establish. By the expression \eqref{keyObs1}, we can understand the behavior of all singular values of $\mathcal{T}_n$ by the singular values of $\mathcal{S}_{n}$. In particular, we have, \[\sigma^{2}_{\min}\lP \mathcal{C}_{2n}\rP\sigma_{\min}\lP S_n\rP \leq \sigma_{\min}\lP \mathcal{T}_n\rP.\] The term $ \sigma^{2}_{\min}\lP\mathcal{C}_{2n}\rP$ is well understood by Theorem \ref{thm01042018} and Theorem \ref{thm:LowestSingV}. Our conjecture is that $S_n$ is well conditioned. In Figure \ref{fig020920201336}, we observe the results of a random experiment of 10,000 simulations to estimate $\sigma_{\min}(S)$, when the coefficients of $G_{2n}$ are Bernoulli $0-1$ with parameter $\frac{1}{2}$, Rademacher $\pm1$ with parameter $\frac{1}{2}$, Uniform on $(0,1)$, or Standard Normal, taking $2n=2048$. We observe that $\sigma_{\min}\lP S_n\rP$ has a nice behavior. In Table \ref{tab060820202} give a resume of the results of the random experiment for the four different distributions.\\
}

\begin{table}[h]
\begin{center}
\caption{Minimum Singular Value of matrix $S$}\label{tab060820202}
\begin{tabular}{|c|c|c|c|}\hline
Distribution & Min & Mean & 1st Quantil\\ \hline
\multicolumn{1}{|l|}{Bernoulli ($1/2$) $0-1$} &  \multicolumn{1}{|l|}{0.0000085}& \multicolumn{1}{|l|}{0.4676160} & \multicolumn{1}{|l|}{0.1778511}\\ \hline
\multicolumn{1}{|l|}{Rademacher ($1/2$) $\pm1$} & \multicolumn{1}{|l|}{0.0000006} & \multicolumn{1}{|l|}{0.2335359} & \multicolumn{1}{|l|}{0.0882868} \\ \hline
\multicolumn{1}{|l|}{Uniform $(0,1)$} & \multicolumn{1}{|l|}{0.000023} & \multicolumn{1}{|l|}{0.796888} & \multicolumn{1}{|l|}{0.293408} \\ \hline
\multicolumn{1}{|l|}{Standard Normal} & \multicolumn{1}{|l|}{0.0001075} & \multicolumn{1}{|l|}{0.2279807} &  \multicolumn{1}{|l|}{0.0859852}\\ \hline
\end{tabular}
\end{center}
\end{table}

\begin{center}
\begin{figure}\label{fig020920201336}
\scalebox{1}{\includegraphics{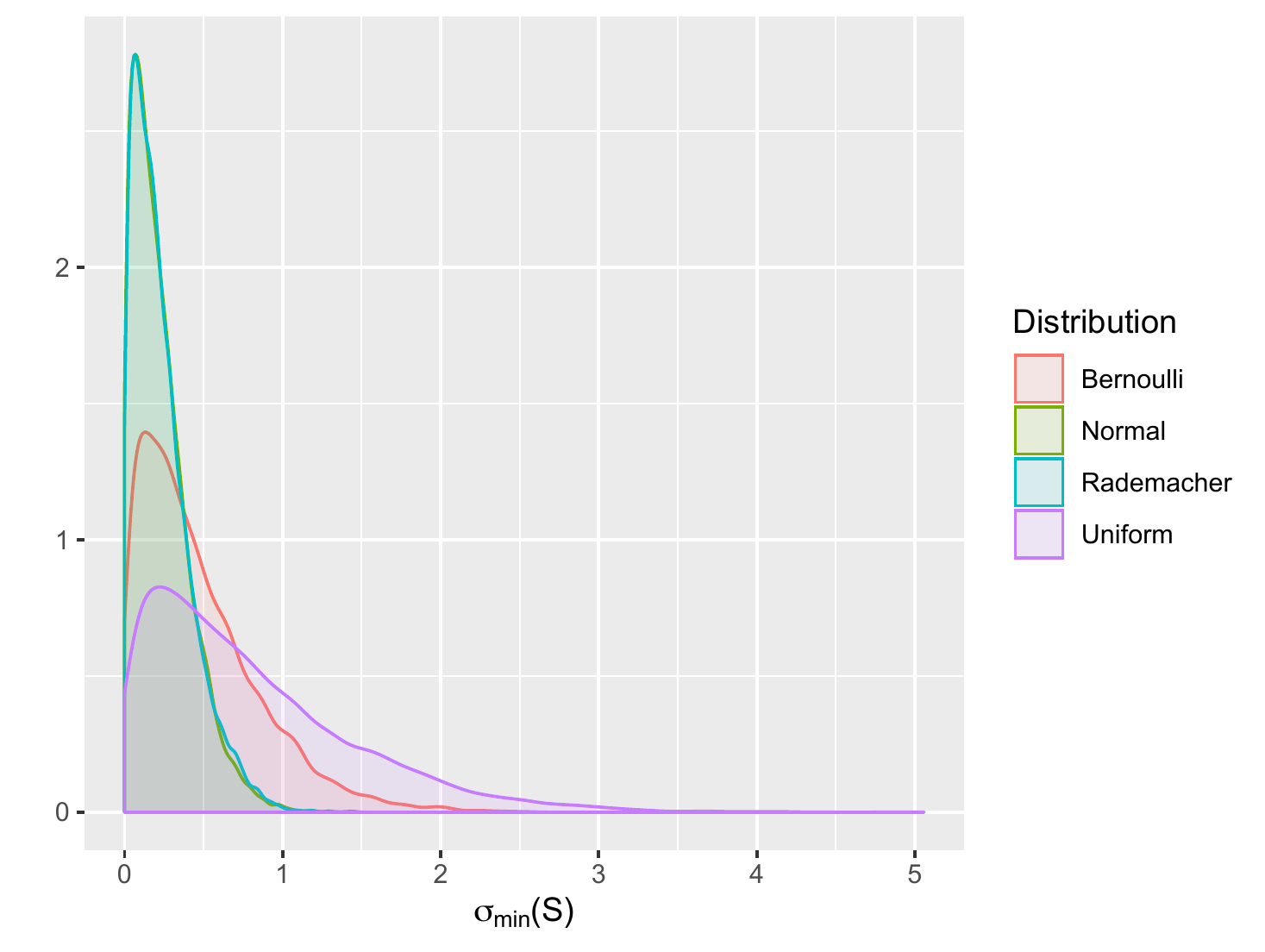}}
\caption{The approximation of density for $\sigma_{\min}(S_n)$ by 10,000 simulations with $2n=2048$ for Bernoulli, Rademacher, Uniform, and Standard Normal distributions.}
\end{figure}
\end{center}

\section{Random circulant matrices}\label{sec:Circulant} 

We assume the entries of a circulant matrix $\mathcal{C}_n$ are r.v. in $\Xi$. In the following, we estimate $\sigma_{\min}\lP\mathcal{C}_n\rP$ and $\sigma_{\max}\lP\mathcal{C}_n\rP$ when $\mathcal{C}_n$ is a random non--symmetric or symmetric matrix.\\

Since the eigenvalues of any circulant matrix $\mathcal{C}_n$ are $G_n\lP w_n^k \rP$, $k=0,\ldots,n-1$, we have
\begin{equation}\label{eqn1502020}
\sigma_{\max}(\mathcal{C}_n) =\max_{k=0,\ldots,n-1} \abs{\lambda_k} \leq \max_{z\in\C: \abs{z}=1} \abs{G_n(z)}.
\end{equation}

We will take advantage of the relationship \eqref{eqn1502020} to estimate $\sigma_{\max}(\mathcal{C}_n)$. As the coefficients of $G_n$ are the entries of the first row of $\mathcal{C}_n$, if $\mathcal{C}_n$ is a random matrix, then $G_n$ is a random polynomial with i.i.d. coefficients in $\Xi$. Thus, we need to estimate the maximum modulus of a random complex polynomial on the unit circle. This problem is interesting in itself. In fact, it has been studied for a long time; for example see \cite{erdos,Kah1985,Weber2006}. Any upper bound of the maximum modulus of a random polynomial on the unit circle is usually called {\it Salem--Zygmund inequality} \cite{Kah1985}. In the following statements are established Salem--Zygmund inequalities when the random coefficients are i.i.d. r.v. with m.g.f. or they have only second moment.\\

Since we are interested in the maximum modulus of $G_n(z)=\sum_{j=0}^{n-1}Ê\xi_j z^j$ on the unit circle, we can consider $W_n(x):=G_n\lP e^{ix}\rP$ for $x\in \mathbb{T}$, where $\mathbb{T}$ denotes the unit circle $\R/(2\pi\Z)$. In this way, the maximum modulus of $G_n$ on the unit circle is denoted by $\| W_n\|_\infty$.

\begin{theorem}[Salem--Zygmund inequality type I]\label{lem06082016m02}
{If} $\xi_0$ has zero mean, finite positive variance, and its m.g.f. exists in an open interval around zero. Then,
\[
\Prob{ \|W_n\|_\infty \geq C_0 \lP n\log n\rP^{1/2} }\leq \frac{C_1}{n^2},
\] where $C_0$ and $C_1$ depend only on the distribution of $\xi_0$.
\end{theorem}

From the expression \eqref{eqn1502020} and Theorem \ref{lem06082016m02} we get:
\begin{equation}\label{eqn:15020201422}
\Prob{\sigma_{\max}(\mathcal{C}_n) \geq C_0 \lP n\log n\rP^{1/2} } \leq \Prob{ \|W_n\|_\infty \geq C_0 \lP n\log n\rP^{1/2} }\leq \frac{C_1}{n^2}.
\end{equation}
\begin{remark} By the Borel--Cantelli Lemma, we deduce from \eqref{eqn:15020201422} that, \[\limsup_{n\to\infty} \frac{\sigma_{\max}(\mathcal{C}_n)}{\lP n\log n\rP^{1/2}} \leq C_0\;\;\; \mbox{ almost surely,} \] and also we have an equivalent statement for the maximum modulus of a random polynomial on the unit circle, 
\[\limsup_{n\to\infty} \frac{ \|W_n\|_\infty}{\lP n\log n\rP^{1/2}} \leq C_0\;\;\; \mbox{ almost surely. }\]
\end{remark}

The conditions in Theorem \ref{lem06082016m02} can be relaxed up to the condition that the random coefficients are i.i.d. r.v. with only zero mean and finite second moment. For this, we use the expectation of the maximum modulus of random polynomial on the unit circle. More precisely, Weber \cite{Weber2006} shows that,
\begin{align}
\mathbb{E}\left(\max\limits_{x\in \mathbb{T}}\left|\sum_{j=0}^{n-1} \xi_j e^{ijx}\right|\right)\leq & \; C
\min\left\{(n\log(n+1)\mathbb{E}{(|\xi_0|^2)})^{\nicefrac{1}{2}},n\mathbb{E}{|\xi_0|}\right\} \nonumber \\
\leq &\; C (n\log(n+1)\mathbb{E}{(|\xi_0|^2)})^{\nicefrac{1}{2}},
\end{align}
where $C$ is a universal positive constant. Hence, using Markov inequality, we can deduce that,
\begin{eqnarray}\label{eqn:15020201436}
\Prob{\sigma_{\max}(\mathcal{C}_n) \geq C_0 n^{1/2}\lP \log n\rP } & \leq & \Prob{\|W_n\|_\infty \geq C_0 n^{1/2}\lP \log n\rP } \nonumber \\
& \leq & \frac{C\lP\E{\xi_0^2}\rP^{1/2} \lP n \log n\rP^{1/2}}{ C_0 n^{1/2}\log n} \nonumber \\
& = & \frac{C_1}{\lP \log n\rP^{1/2}},
\end{eqnarray} where $C_1$ depends on the distribution of $\xi_0$.\\

For a random symmetric circulant, we establish results equivalent to Theorems \ref{lem06082016m02}. Observed that if $\mathcal{C}_n^{sym}$ is a random symmetric circulant matrix, half of its entries in the first row are i.i.d. First, we established the corresponding Salem--Zygmund inequality for a trigonometric random polynomial where the coefficients of the terms $z^j$ and $z^{n-j}$ are equal.

\begin{theorem}[Salem--Zygmund type II]\label{thm:LargestSingV} Suppose $\xi_0$ has zero mean, finite positive variance, and its m.g.f. exists in an open interval around zero. Let $W^{sym}_n(x) := \sum_{j=0}^{n-1} \xi_j e^{ijx}$ for $x\in\mathbb{T}$ with $\xi_j=\xi_{n-j}\in\Xi$ for $j=1,\ldots, \lfloor n/2\rfloor + a_n$, $a_n=-1$ if $n$ is even and $a_n=0$ if $n$ is odd. If $\|W^{sym}_n\|_\infty:= \max_{x\in\mathbb{T}} \abs{W^{sym}_n(x)}$. Then,
\[
\Prob{ \|W^{sym}_n\|_\infty \geq C_0 \left( n \log n\right)^{\nicefrac{1}{2}} }\leq \frac{C_1}{n^2},
\] where $C_0, C_1$ depend on the distribution of $\xi$.
\end{theorem}

We denote the eigenvalues of $\mathcal{C}^{sym}_n$ by $\lambda^{sym}_k$ for $k=0,\ldots,n-1$. From Theorem \ref{thm:LargestSingV} it is observed that,
\begin{equation}\label{10042020}
\Prob{\max_{k=0,\ldots,n-1} \abs{\lambda^{sym}_k} \geq C_0 \lP n \log n\rP^{1/2}} \leq \Prob{\|W^{sym}_n\|_\infty \geq C_0 \lP n \log n\rP^{1/2} } \leq \frac{C_1}{n^2}.
\end{equation} Thus, $\sigma_{\max}\lP\mathcal{C}^{sym}_n\rP$ is at most $ C_0 \lP n \log n\rP^{1/2}$ with probability  $1-\mbox{O}\lP n^{-2}\rP$.

\begin{remark}
If the random entries of a symmetric circulant matrix are Gaussian, Adhikari and Saha \cite{adhikari2017fluctuations} show that $\limsup_{n\to\infty} \frac{\sigma_{\max}\lP\mathcal{C}^{sym}_n\rP}{\sqrt{n\log n}} \leq C_0$ almost surely for some constant $C_0$. Actually, they mention that this result holds for sub--Gaussian r.v.  But, using our result from random polynomials and Borel-Cantelli Lemma, the same result holds for random variables with m.g.f. Moreover, we have, \[\limsup_{n\to\infty}\frac{\| W^{sym}_n\|_\infty}{\sqrt{ n \log n }} \leq C_0\] almost surely.
\end{remark}

From relationships in \eqref{keyObs} we give an estimation the condition number of a random rectangular Toeplitz matrix. In the case a rectangular matrix $A\in\R^{m\times n}$, its condition number is defined as \cite[Corollary 1.27]{burgisser2013condition}:
\begin{equation}\label{eqn250820200732}
\kappa(A)=\frac{\sigma_{\max}}{d(A,\Sigma)},
\end{equation} where $\Sigma=\lL B\in\R^{m\times n} : \textnormal{rank}(B) <\min\lL n,m\rL\rL$. Additionally, if $m\geq n$, we have $\min_{\absD{x}=1} \absD{Ax} = \sigma_n(A)$ (see \cite[Proposition 1.15]{burgisser2013condition}). Thus, we have the following result.
\begin{theorem}\label{cor250820200741} Let $\mathcal{A}_n:=\left[ 
\begin{array}{c}
\mathcal{T}_n \\
\mathcal{B}_n
\end{array}
\right]$, we have:
\begin{itemize}
\item Suppose $\mathcal{T}_n$ is a non--symmetric random Toeplitz, and $\xi_0$ has zero mean and finite positive variance. If the m.g.f. of $\xi_0$ exists in an open interval and $\xi_0$ satisfies Condition \eqref{fanto1}. Then, for any $\varepsilon>0$ and $\rho\in(0,1/4)$, we have  for all large $n$,
\[
 \Prob{\kappa\lP\mathcal{A}_n\rP \leq \frac{C_0}{\varepsilon} n^{\rho +1/2} \lP\log n\rP^{1/2}} \geq 1-C_1\lP \lP\varepsilon^2+\varepsilon\rP n^{-2\rho} + n^{-1/2+\textnormal{o}(1)}\rP,
\] where $C_0, C_1$ depend on the distribution of $\xi_0$.

\item Suppose $\mathcal{T}_n$ is a symmetric random Toeplitz, and $\xi_0$ has zero mean and finite positive variance. If the m.g.f. of $\xi_0$ exists in an open interval and $\xi_0$ satisfies Condition \eqref{fanto1}. Then, for any $\varepsilon>0$, we have for all large $n$,
\[
\Prob{\kappa\lP \mathcal{A}_n\rP \leq \frac{C_0}{\varepsilon} n^{1.01} \lP\log n\rP^{1/2}} \geq 1-C_1\lP \varepsilon n^{-0.1} + n^{-77/300+\textnormal{o}(1)}\rP,
\] where $C_0, C_1$ depend on the distribution of $\xi_0$. 
\end{itemize}
\end{theorem}

\section{Proof of the Salem--Zygmund inequalities}\label{sec:salemZygmund}

The strategy to prove Theorem \ref{lem06082016m02} is essentially given in \cite{Barrera}, see the proof of Theorem 1.2 therein. Here, we give the outline of the proof, but since the proof of Theorem \ref{thm:LargestSingV} follows similar ideas, we include details of our arguments in Appendix \ref{app10042020}.\\

The existence of m.g.f. of $\xi_0$ around zero implies that for any $x\in\mathbb{T}$,
\[
\E{e^{t W_n(x)}} \leq  e^{\nicefrac{\alpha^2 t^2 n}{2}},
\] for some fixed positive constant $\alpha^2$ which depends on the distribution of $\xi_0$. It is possible to show that there exists an interval $I\subset \mathbb{T}$ such that $|W_n(x)|\geq \frac{1}{2}\|W_n\|_\infty$ for $x\in I$ and the length of $I$ is $\frac{8}{3n}$. Then,
\[
\E{\exp\lP \frac{1}{2}t\|W_n\|_\infty\rP}  \leq  \frac{8n}{3} \E{\int_I\lP e^{t W_n(x)}+ e^{-tW_n(x)}\rP \mu(dx)} 
 \leq  \frac{16 n}{3} \exp\lP 3\alpha^2 t^2 n/2\rP.
\]

Finally, we use Chernoff bounding technique to obtain an upper bound for $\|W_n\|_\infty$ with high probability. Let $b_n$ be a positive real number for $n\in\N$. Then,
\[
\Prob{\| W_n\|_\infty \geq b_n} = \Prob{e^{t\| W_n\|_\infty} \geq e^{tb_n}} \leq e^{-t b_n} \E{e^{t\| W_n\|_\infty}}.
\] In Appendix \ref{app10042020} we show how to select  adequate $t$ and $b_n$ such that,
\[
\Prob{ \|W_n\|_\infty \geq C_0 \left( n \log n\right)^{\nicefrac{1}{2}} }\leq \frac{C_1}{n^2}.
\]

\section{Proof of Theorem \ref{thm01042018}} \label{sec:LowestSingV1}

Before starting, we need to introduce little notations. The {\it floor} of a real number $x$, denoted by $\lfloor x \rfloor$, is the greatest integer $n$ such that $n\leq x$. For $k,l\in\Z$, the greatest common divisor of $k$ and $l$ is denoted by $\gcd\lP k,l\rP$. Recall that $\|\cdot\|_2$ is the Euclidean norm in $\R^n$. The determinant of a square matrix is denoted by $\det(\cdot)$. Let $f_n,g_n$ be two real sequences, we write $f_n=\mbox{o}(g_n)$ if for every $\alpha>0$ there exists $n_0\in\N$ such that for all $n\geq n_0$ we have $\abs{f_n}\leq \alpha \abs{g_n}$. \\

The target is to find a nice upper bound of the L\'evy concentration of $n^{\rho}\sigma_{\min}\lP\mathcal{C}_n\rP$. Recall that the eigenvalues of a circulant matrix are $\lambda_k=G_n\lP e^{i2\pi n/k}\rP$, $k=0,\ldots,n-1$, where $G_n(z)=\sum_{j=0}^{n-1} \xi_j z^j$ and $\xi_0,\ldots,\xi_{n-1}$ are i.i.d. r.v. in $\Xi$. If $x_k := \nicefrac{k}{n}$, we have,
\begin{eqnarray}\label{eqn:1802202011445}
& & \hspace{-1cm}\Prob{\sigma_{\min}(\mathcal{C}_n) \leq \varepsilon n^{-\rho}} =  \Prob{n^{\rho}\sigma_{\min}(\mathcal{C}_n) \leq \varepsilon } \leq  \sum_{k=0}^{n-1} \Prob{ \abs{n^{\rho} G_n(e^{i2\pi x_k})} \leq \varepsilon} \nonumber \\
& & \hspace{-0.5cm} \leq \mathcal{L}\lP n^{\rho} G_n(1),\varepsilon\rP + \mathcal{L}\lP n^{\rho} G_n(-1),\varepsilon\rP + \displaystyle\sum_{\substack{k=0 \\ k\neq 0,n/2}}^{n-1} \mathcal{L}\lP n^{\rho} G_n\lP e^{i2\pi x_k}\rP,\varepsilon \rP. 
\end{eqnarray}
Note that in the expression \eqref{eqn:1802202011445}, $n^{\rho}G_n(e^{i2\pi x_k})$ is a sum of r.v. with (deterministic) real or complex coefficients for all $k=0,\ldots, n-1$. To estimate the L\'evy concentration for each of these sums, we use the {\it least common denominator}, which is defined as follows.

\begin{definition}\label{def20042020}
Let $L$ be any fixed positive number. The {\it least common denominator} (LCD) of a matrix $V\in\R^{m\times n}$ is defined as, 
\[
{D(V)=D(V,L)}:= \inf \lL \|\theta \|_2 >0 : \theta\in\R^m, \textnormal{dist}\lP V^T\theta,\Z^n \rP < L\sqrt{\log_{+} \left(\frac{\| V^T\theta\|_2}{L}\right)} \rL,
\]
where 
{$\textnormal{dist}(v,\mathbb{Z}^n):=\min\lL \| v - z\|_2 : z\in\Z^n\rL$}
and $\log_{+}(x)=\max\{\log(x),0\}$.
\end{definition}

Note that if $v\in\R^n$, the definition of LCD for $v$ reduces to, \[
D(v) = D(v,L) = \inf\lL \theta>0 : {\dist(\theta v,\Z^n)} <L \sqrt{\log_+\frac{\|\theta v\|_2}{L}}\rL.
\]

The notion of LCD used here was introduced by Rudelson and Vershynin \cite{RV2} in the study of the eigenvectors of random matrices with all independent random entries. For a given matrix $V$, we denote by $\| V\|_\infty$ the maximum Euclidean norm of the columns of $V$. Rudelson and Vershynin establish a simple lower bound for LCD of $V$ in terms of $\| V\|_\infty$ (see Proposition 7.4 in \cite{RV2}).

\begin{proposition}[Simple lower bound for LCD]\label{pro23042020} For every matrix $V$ and $L>0$, one has \[ D(V,L) \geq \frac{1}{2\| V\|_\infty}. \] 
\end{proposition}

Moreover, they show how to relate the L\'evy concentration function with the LCD (see Theorem 7.5 in \cite{RV2}).

\begin{theorem}[Small Ball Probability Inequality]\label{thm:small} If $V$ is an $m\times n$ matrix  and $X\in\R^n$ is a random vector with i.i.d. entries which satisfy condition \eqref{fanto1}. Then for every $L\geq \sqrt{m/q}$ we have, 
\[
{\mathcal{L}\lP V X,\varepsilon \sqrt{m} \rP} \leq \frac{\lP CL/\sqrt{m}\rP^m}{\lP\det\lP VV^T\rP \rP^{1/2}} \lP \varepsilon + \frac{\sqrt{m}}{D(V)}\rP^m, \quad \varepsilon \geq0.
\] The constant $C$ depends on the distribution of the entries of $X$, and $D(V)$ is the LCD of $V$.
\end{theorem}

A special case of Theorem \ref{thm:small} is when $m=1$. In this case, $V^T X$ represents a sum of r.v.

\begin{corollary}[Small ball probabilities for sums]\label{thm060120201236} Let $\xi_k$ be i.i.d. copies of $\xi$ satisfying condition \eqref{fanto1}. Let ${\bf a}=(a_1,\ldots,a_n)\in\R^n$. Then for every $L\geq (1/q)^{1/2}$, we have,
\[
\mathcal{L}\lP \sum_{k=1}^n a_k\xi_k, \varepsilon\rP \leq \frac{CL}{\|{\bf a}\|_2}\lP \varepsilon + \frac{1}{D({\bf a},L)}\rP, \;\;\; \varepsilon\geq 0.
\] The constant $C$ depends on the distribution of $\xi$, and $D({\bf a},L)$ is the LCD of ${\bf a}$.
\end{corollary}

Now, we can proceed to give an upper bound of $\mathcal{L}\lP n^{\rho} G_n\lP e^{i2\pi x_k}\rP, \varepsilon\rP$. Actually, this is shown by Barrera and Manrique in \cite{Barrera}, see Theorem 1.6. But for the sake of clarity, we decided to include here the important parts of the used strategy. \\

To apply Theorem \ref{thm:small}, we distinguish two cases given by the expression \eqref{eqn:1802202011445}, when $n^{\rho}G_n(e^{i2\pi x_k})$ has real or complex coefficients.  

\begin{lemma}[Real coefficients]\label{lem:20022020830} Under the hypothesis of Theorem \ref{thm01042018}, we have for any $\varepsilon\geq0$,
\[
\mathcal{L}\lP n^{\rho} G_n(1),\varepsilon\rP + \mathcal{L}\lP n^{\rho} G_n(-1),\varepsilon\rP \leq C_1\lP\frac{ \varepsilon}{n^{\rho+1/2}} + \frac{1}{n^{1/2}}\rP,
\] where $C_1$ depends on the distribution of $\xi_0$.
\end{lemma}
\begin{proof} 
Note $n^{\rho} G_n(1) = n^{\rho} \sum_{j=0}^{n-1}\xi_j$. Write $\mathbf{a}=(1,\ldots,1)^{T}\in\R^n$. By Proposition \ref{pro23042020}, we have the LCD of $\mathbf{a}$ is such that  $D(\mathbf{a})\geq \frac{1}{2}n^{-\rho}$. Let $L\geq (1/q)^{1/2}$ a fixed number. By Corollary \ref{thm060120201236}, we have,
\[
\mathcal{L}\lP n^{\rho} G_n(1),\varepsilon\rP \leq \frac{CL}{n^{\rho}n^{1/2}}\lP \varepsilon + 2n^{\rho}\rP \leq C_1\lP\frac{\varepsilon}{n^{\rho+1/2}} + \frac{1}{n^{1/2}}\rP,
\] where $C_1$ depends on the distribution of $\xi_0$.

For $n^{\rho} G_n(-1) = n^{\rho} \sum_{j=0}^{n-1}(-1)^j\xi_j$, the proof is similar, but taking $\mathbf{a}=(1,-1,\ldots,(-1)^n)\in\R^n$.
\end{proof}

For our second case with complex coefficients, we define the $2\times n$ matrix $V_k$, $k=0,\ldots,n-1$, as 
\begin{equation}\label{eqn:20022020922}
V_k := 
\lC
\begin{array}{cccc}
1 & \cos\lP 2\pi x_k \rP & \ldots & \cos\lP (n-1)2\pi x_k \rP \\
0 & \sin\lP 2\pi x_k \rP & \ldots & \sin\lP (n-1)2\pi x_k \rP
\end{array}
\rC,
\end{equation} where $x_k=\frac{k}{n}$. Let $X:=\lC\xi_0,\ldots,\xi_{n-1}\rC^T \in \R^n$. Then,
\[
V_k X = \lC
\sum_{j=0}^{n-1} \xi_j \cos\lP j2\pi x_k\rP, \sum_{j=0}^{n-1} \xi_j \sin\lP j2\pi x_k\rP \rC^T \in \R^2,
\] which implies
\begin{equation}\label{eqn200220201218}
\| V_k X\|_2 = \abs{\sum_{j=0}^{n-1} \xi_j e^{ij2\pi x_k}} = \abs{G_n\lP e^{i2\pi x_k}\rP}.
\end{equation}
On the other hand, we have for all $k$,
\begin{equation}\label{eqn:20022020824}
\det\lP V_k V_k^T \rP = \det\lC
\def\arraystretch{1.5}
\begin{array}{cc}
\sum_{j=0}^{n-1} \cos^2\lP j2\pi x_k \rP & \frac{1}{2}\sum_{j=0}^{n-1} \sin\lP 2\cdot j2\pi x_k \rP \\
\frac{1}{2}\sum_{j=0}^{n-1} \sin\lP 2\cdot j2\pi x_k \rP & \sum_{j=0}^{n-1} \sin^2\lP j2\pi x_k \rP
\end{array} 
\rC= \frac{n^2}{4}.
\end{equation}

Before continuing, we need to introduce two auxiliary lemmas, which can be found in \cite{Barrera}, but for the sake of clarity, we include their proofs in Appendix \ref{app30042020}. The first lemma gives an upper bound for the number of positive integers $l$ such that $\gcd\lP l,n\rP\geq \alpha$ for some fixed positive $\alpha\in\R$. The second lemma gives a lower bound for the distance between a vector whose entries are cosine values to the grid $\Z^n$.

\begin{lemma}\label{lem270120191754}
Let $y,M\in [1,\infty)$ be fixed numbers.
The cardinality of the set
\[
\left\{k\in [1,M]\cap \mathbb{N}:~\gcd\lP k, M\rP \geq y\right\}
\] is at most $\frac{1}{\lfloor y\rfloor} M^{1+C\lP \log \log M\rP^{-1}}$, where $C$ is a universal positive constant.
\end{lemma}

\begin{lemma}\label{lem050320191745}
Fix $\theta\in[0,2\pi)$ and positive $m\in\Z$. Let $\mathcal{V}$ be a vector in $\R^m$ {whose} entries are $\mathcal{V}_j= r\cos\lP j 2\pi x-\theta \rP$ for $j=0,\ldots,m-1$ with positive integer $r\geq 2$ and $x=\nicefrac{1}{m}$. Then,
\[
\mathrm{dist}\lP \mathcal{V},\Z^m\rP \geq \frac{1}{48}\cdot\frac{1}{2\pi x},
\] whenever $\frac{1}{2r\lP2\pi x\rP}\geq 6$.
\end{lemma}

\begin{lemma}[Complex coefficients]\label{lem:20022020836} Under the hypothesis of Theorem \ref{thm01042018}, we have for all large $n$,
\[
\displaystyle\sum_{\substack{k=0 \\ k\neq 0,n/2}}^{n-1} \mathcal{L}\lP n^{\rho} G\lP e^{i2\pi x_k}\rP,\varepsilon \rP \leq C\lP\frac{\varepsilon^2}{n^{2\rho}} + \frac{1}{n^{1/2-\textnormal{o}(1)}}\rP,
\] where $C$ depends on the distribution of $\xi_0$.
\end{lemma}
\begin{proof}
Recalling that $x_k=\frac{k}{n}$. We need to distinguish two cases for $\gcd(k,n)$. First, we assume $\gcd(k,n)\geq n^{1/2}$. Then by Lemma \ref{lem270120191754}, the number of integers $k$ that satisfies this condition is at most $n^{1/2+\textnormal{o}(1)}$. Note that if $V_k$ is the matrix defined by \eqref{eqn:20022020922}, then by Proposition \ref{pro23042020}, the LCD of $V_k$ satisfies $D(n^{\rho}V_k)\geq \frac{1}{2}n^{-\rho}$. Thus, using the expressions \eqref{eqn200220201218} and \eqref{eqn:20022020824}, by Theorem \ref{thm:small} we have,
\begin{eqnarray}\label{eqn03052020_1}
\sum_{\substack{k=0 \\ \alpha\;:\;  k\neq 0,n/2,\;\; \gcd\lP k, n\rP \;\geq\; n^{1/2}}}^{N-1}  \mathcal{L}\lP n^{\rho} G\lP e^{i2\pi x_k}\rP,\varepsilon \rP 
& \leq & n^{1/2+\textnormal{o}(1)} \frac{C^2L^2}{2n^{1+2\rho}}\lP\varepsilon + 4n^{\rho}\rP^2 \nonumber\\
 & \leq & C_1\lP\frac{\varepsilon^2}{n^{1/2+2\rho-\textnormal{o}(1)}} + \frac{1}{n^{1/2-\textnormal{o}(1)}} \rP,
\end{eqnarray}
where in the last inequality we use the fact $(a+b)^2\leq 2a^2+2b^2$ and $C_1$ is a positive constant depending on the distribution of $\xi_0$.\\

Now, we assume $\gcd(n,k)\leq n^{1/2}$. Let $V_k$ be the matrix defined by \eqref{eqn:20022020922} and $x=\nicefrac{k}{n}$. Let $\Theta = r\lC \cos(\theta), \sin(\theta)\rC^T\in\R^2$, where $r>0$ and $\theta\in\lC0,2\pi\rC$. For fixed $r,\theta$, we have,
\begin{equation}\label{eqn23042020}
V_k^T\Theta = r\lC\cos\lP-\theta\rP,\cos\lP2\pi x - \theta\rP,\ldots,\cos\lP2\pi\lP n-1\rP  x - \theta \rP\rC^T.
\end{equation}
Note that $\|\Theta\|_2=r$, $\| V_k^T\Theta\|_2\leq r\sqrt{n}$, and $V_k^T\Theta$ can have duplicate entries. Now, we need to estimate the LCD of $V_k$. Since $\gcd(n,k)\leq n^{1/2}$, we can apply Lemma \ref{lem050320191745} to $n' = \frac{n}{\gcd(n,k)}\geq n^{1/2}$ and $k'=\frac{k}{\gcd(n,k)}$.  
{
To see this, we observe that $\gcd(n',k')=1$ and $n'\leq n$. Then for any $\theta$,
\begin{equation}\label{eqn1907202020}
\lL\exp\lP i\lP 2\pi j\frac{k'}{n'} -\theta\rP\rP :  j =0,\ldots, n'-1\rL = \lL\exp\lP i\lP 2\pi j\frac{1}{n'} -\theta\rP\rP :  j =0,\ldots, n'-1\rL.
\end{equation} From the above observation, we can assume that $x=1/n'$ and we obtain a lower bound of $\dist\lP n^{\rho} V^T_k \Theta, \Z^n\rP$ by only considering the vector \[r\lC\cos\lP-\theta\rP,\cos\lP2\pi \frac{1}{n'} - \theta\rP,\ldots,\cos\lP2\pi\lP n'-1\rP  \frac{1}{n'} - \theta \rP\rC^T.\] 
}

{
Taking into account the expressions \eqref{eqn23042020}, \eqref{eqn1907202020}, and the previous comments. If $n^{\rho}r\leq \frac{n^{1/2}}{12\cdot 2\pi}$, by Lemma \ref{lem050320191745} and the definition of LCD we have,
\[
\frac{1}{48}\cdot\frac{n^{1/2}-1}{2\pi} \leq \dist\lP n^{\rho} V^T_k \Theta, \Z^n\rP < L\sqrt{\log_+\frac{\|n^{\rho} V^T_k\Theta \|_2}{L}} \leq L\sqrt{\log_+\frac{n}{L}},
\] which is a contradiction for all large $n$ if we assume $L$ is fixed. Thus, the LCD of $n^{\rho} V^T_k$ should satisfy, \[D\lP n^{\rho} V_k \rP \geq \frac{n^{1/2-\rho}}{48\pi}\] for all large $n$.
}

Notice that here we assume that $r$ is a positive integer. Actually, by Proposition \ref{pro23042020} we can assume that $r\geq 1/2$. 
{
If $r\geq 2$, we can use $\lfloor r \rfloor$ instead $r$ to apply Lemma \ref{lem050320191745}. In fact, in the proof of Lemma \ref{lem050320191745} in Appendix \ref{app30042020}, the used arguments can be applied to $r$ with the small difference that the condition $\frac{1}{12\lP2\pi x\rP}\geq r$ changes to $\frac{1}{12\lP2\pi x\rP}+1\geq r$, since the sum in expression \eqref{eqn280221091800} we can replace $r$ by $\lfloor r \rfloor$ and observing $\lfloor r \rfloor\geq r-1$. To handle $2>r\geq 1/2$, we notice from expression \eqref{eqn23042020} that we can apply Lemma \ref{lem050320191745} using $\lfloor n^{\rho}r\rfloor$, which is an integer bigger than $2$ for all sufficiently large $n$, whenever it satisfies $\frac{n'}{24\pi}\geq \lfloor n^{\rho}r\rfloor$, but this holds since $\frac{n'}{24\pi}\geq \frac{\sqrt{n}}{24\pi} \geq n^{\rho}r$ with $2>r\geq 1/2$ and $\rho\in(0,1/4)$. Hence, if $2>r\geq 1/2$, $\dist\lP n^{\rho} V^T_k \Theta, \Z^n\rP\geq C n^{1/2}$ for some positive constant $C$.
 }
Then, by Theorem \ref{thm:small} (recall $V_k\in\R^{2\times n}$) and expression \eqref{eqn:20022020824}, we have,
\begin{eqnarray}\label{eqn03052020_2}
 \hspace{-1cm}\displaystyle\sum_{\substack{k=0 \\ \alpha\;:\;  k\neq 0,n/2,\;\; \gcd\lP k, n\rP \;\leq\; n^{1/2}}}^{N-1}  \mathcal{L}\lP n^{\rho} G\lP e^{i2\pi x_k}\rP,\varepsilon \rP  & \leq  & n \cdot \frac{C^2L^2}{2n^{1+2\rho}}\lP\varepsilon + \frac{48\pi}{ n^{1/2-\rho}}\rP^2 \nonumber\\
 &\leq & C_2\lP\frac{\varepsilon^2}{n^{2\rho}} + \frac{1}{n} \rP,
\end{eqnarray} where the positive constant $C_2$ depends on the distribution of $\xi_0$. Thus, from \eqref{eqn03052020_1} and \eqref{eqn03052020_2} we have for all large $n$,
\begin{eqnarray*}
\displaystyle\sum_{\substack{k=0 \\ k\neq 0,n/2}}^{n-1} \mathcal{L}\lP n^{\rho} G\lP e^{i2\pi x_k}\rP,\varepsilon \rP & \leq & C_1\lP\frac{\varepsilon^2}{n^{1/2+2\rho-\textnormal{o}(1)}} + \frac{1}{n^{1/2-\textnormal{o}(1)}} \rP + C_2\lP\frac{\varepsilon^2}{n^{2\rho}} + \frac{1}{n} \rP \\
& \leq & C_3\lP \frac{\varepsilon^2}{n^{2\rho}} + \frac{1}{n^{1/2-\textnormal{o}(1)}}\rP,
\end{eqnarray*} where the positive constant $C_3$ depends on the distribution of $\xi_0$.
\end{proof}

\noindent {\bf Proof Theorem \ref{thm01042018}.} By expression \eqref{eqn:1802202011445} and Lemmas \ref{lem:20022020830} and \ref{lem:20022020836}, for any $\varepsilon\geq 0$ and $\rho\in(0,1/4)$ we have for all large $n$,  
\begin{eqnarray*}
\Prob{\sigma_{\min}\lP \mathcal{C}_n\rP \leq \varepsilon n^{\rho}} & \leq & C_1\lP\frac{ \varepsilon}{n^{\rho+1/2}} + \frac{1}{n^{1/2}}\rP +  C\lP\frac{\varepsilon^2}{n^{2\rho}} + \frac{1}{n^{1/2-\textnormal{o}(1)}} \rP\\
& \leq & C_2\lP \frac{\varepsilon^2+\varepsilon}{n^{2\rho}} + \frac{1}{n^{1/2-\textnormal{o}(1)}}\rP,
\end{eqnarray*} where the positive constant $C_2$ depends on the distribution of $\xi_0$.

\section{Proof of Theorem \ref{thm:LowestSingV}}\label{sec:minSinSymCir}

In this section, we give the proof of Theorem \ref{thm:LowestSingV}. Again, we use LCD to give a nice upper bound of the probability of the event $\lL \min_{k} \abs{\lambda^{sym}_k} \leq \varepsilon n^{-0.51} \rL$. To do this, we need to observe a useful property of the L\'evy concentration of the sum of independent r.v. Its proof is immediate from Definition \ref{def06012020}.

\begin{proposition}\label{pro:060120200920}
Let $\varepsilon \geq 0$. If $X,Y\in\R$ are independent random variables then
\begin{equation*}
\mathcal{L}\lP X+Y, \varepsilon\rP \leq \min\lP \mathcal{L}\lP X, \varepsilon\rP, \mathcal{L}\lP Y, \varepsilon\rP\rP.
\end{equation*}
\end{proposition}

From Proposition \ref{pro:060120200920} and expressions \eqref{eqn:1453} and \eqref{eqn:1454}, we observe for any large $n$ that,
\begin{eqnarray}\label{eqn29042020}
\Prob{\min_{0\leq k\leq \lfloor n/2\rfloor} \abs{\lambda_k} \leq \varepsilon n^{-0.51}}  & \leq & \Prob{\abs{\lambda_0} \leq \varepsilon n^{-0.51}} +  \sum_{k=1}^{\lfloor n/2\rfloor} \Prob{ \abs{\lambda_k} \leq \varepsilon n^{-0.51}} \nonumber \\ 
& \leq &\mathcal{L}\lP n^{0.51} S_0,Ê\varepsilon\rP  + \sum_{k=1}^{\lfloor n/2\rfloor} \mathcal{L}\lP n^{0.51}S_{n,k}, \varepsilon \rP,
\end{eqnarray}
where, \[ S_0:=\sum_{j=1}^{\lfloor n/2\rfloor -1} \xi_j, \;\;\; S_{n,k} :=  \sum_{j=1}^{\lfloor n/2\rfloor -1} \xi_j \cos\lP\frac{2\pi k}{n} j\rP. \]

Let $v\in\R^{\lfloor n/2\rfloor-1}$ with entries $v_j:=\cos\lP \frac{2\pi k}{n} j\rP$ for $j=1,\ldots,\lfloor n/2\rfloor-1$. From Corollary \ref{thm060120201236}, we observe that if the LCD of $v$ is sufficiently large then $\mathcal{L}\lP n^{0.51}S_{n,k},\varepsilon\rP$ will be small. Hence, our problem is reduced to the analysis of the arithmetic structure of $v$. For this, we establish the next lemma.

\begin{lemma}\label{lem060120201249}
Let $n,k$ be fixed positive integers with $\gcd(n,k)=1$ and $n>k$. Let $v$ be the vector in $\R^{\lfloor n/2\rfloor - 1}$ whose entries are $v_j=\cos\lP 2\pi k j x\rP$ for $j=1,\ldots,\lfloor n/2\rfloor-1$ with $x:=\nicefrac{1}{n}$. Then for all large $n$,  \[ \dist\lP rv,\Z^{\lfloor n/2\rfloor-1}\rP \geq \frac{1}{1728\pi x},\] whenever $\frac{1}{36\cdot 2\pi x}\geq r\geq 1$.
\end{lemma}
\begin{proof} Here $i$ is the imaginary unit. Fix $k$ with $\gcd(k,n) = 1$. Note that $\cos\lP 2\pi k j x\rP$ is the real part of $\exp\lP i 2\pi k j x\rP$ for all $j$. The set of points of the form $\exp\lP i 2\pi k j x\rP$ for $j=0,\ldots,n-1$ can be seen as the vertices of a regular polygon $P$ inscribed in the unit circle. The vector $v$ considers at most half of the vertices of this regular polygon of $n$ sides. By the pigeonhole principle, there exists a quadrant $Q$ of the plane where there are at least $\nicefrac{\lfloor n/2 \rfloor}{4}$ vertices of $P$ which are entries of $v$. Note that $\nicefrac{\lfloor n/2 \rfloor}{4} \geq \nicefrac{n}{9}$ for all $n\geq 18$. In the following, we fix the quadrant $Q$ obtained by the pigeonhole principle. Note that the difference between the arguments of adjacent vertices of $P$ in $Q$ {which are entries of $v$} is at most $3\cdot 2\pi x$ for all $n\geq 18$.\\

Let $J:=[-1,1]\cap Q$. Note that $J$ is a close interval, which can be $[-1,0]$ or $[0,1]$. Let $[y,y+9\cdot 2\pi x]$ be a closed interval in $J$. Let $\stackrel{\frown}{A}$ be the arc on the unit circle in the quadrant $Q$ such that its projection in the horizontal axis is $[y,y+9\cdot 2\pi x]$. If the length of $\stackrel{\frown}{A}$ is $l$, then the number of values $\cos\lP 2\pi jkx\rP$ which are in $\lP y, y+9\cdot 2\pi x\rP$ are at least, \[ \frac{l}{3\cdot 2\pi x} - 2 \geq \frac{9\cdot 2\pi x}{3\cdot 2\pi x} - 2 = 1,\] since $l\geq 9\cdot 2\pi x$. \\

Let $I:=\lL j\in\lL 1,...,\lfloor n/2\rfloor-1 \rL :  \cos\lP 2\pi j k x\rP\in J \rL$. Note that $\abs{I}\geq \nicefrac{n}{9}$ for $n\geq 18$. {Fix a positive number $r\leq \frac{n}{9} = \frac{1}{9x}$}. Let $K^r$ be the set of integer $s$ with $\abs{s}<r$ and $\lC \frac{s}{r},\frac{s+1}{r} \rC  \subset J$. Thus, $\abs{K^r}\geq \lfloor r\rfloor$. We take an $s\in K^r$. For each $j\in I$ there exists at least one value, 
\[
\cos\lP 2\pi jkx\rP \in \lP \frac{s}{r} + 9(\alpha -1)(2\pi x), \frac{s}{r} + 9\alpha(2\pi x) \rP \subset \lC \frac{s}{r},\frac{s+1}{r} \rC,
\] for all positive integer $\alpha \leq \frac{1}{r(9\cdot 2\pi x)}$.

Let $I_s^r\subset I$ such that $\cos\lP 2\pi jkx\rP \in \lC \frac{s}{r},\frac{s+1}{r} \rC$ for all $j\in I_s^r$. We define,
\[
d_j := \min \lL \abs{ \cos\lP 2\pi jkx\rP - \frac{s}{r}} , \abs{\cos\lP 2\pi jkx\rP - \frac{s+1}{r}} \rL.
\]
Let $L$ be the biggest integer that satisfies $ \lP 9\cdot 2\pi x\rP L \leq \frac{1}{2r}$, i.e., $L=\left\lfloor \frac{1}{2r\cdot 9 \cdot 2\pi x} \right\rfloor$. Observe that,
\[
L \geq \frac{1}{2r\cdot 9\cdot 2 \pi x} - 1 \geq \frac{1}{2}\lP \frac{1}{2r\cdot 9 \cdot 2\pi x}\rP\; \mbox{ whenever $1\geq 36r \cdot 2\pi x$.}
\]  

Then, 
\begin{eqnarray*}
\sigma_s^r := \sum_{j\in I_s^r} d_j \geq \sum_{\lambda = 1}^L 2\lambda\lP 3\cdot 2\pi x\rP  =  6\cdot 2\pi x\sum_{\lambda =1}^L \lambda & = & 6\cdot 2\pi x\frac{L(L+1)}{2} \\
& \geq  & 3 \cdot 2\pi x L^2 \\
& \geq & 3 \cdot 2\pi x\lP \frac{1}{2}\cdot \frac{1}{2r\cdot 9\cdot 2\pi x}\rP^2 = \frac{1}{4}\cdot\frac{1}{4r^2\cdot 27 \cdot 2\pi x}.
\end{eqnarray*}

Now, we take the sum of all $\sigma_s^r$ with $s\in K^r$,
\[
\sum_{s\in K^r} \sigma^r_s \geq \lfloor r\rfloor \cdot \frac{1}{432}\cdot \frac{1}{r^2\cdot 2\pi x}.
\]
By the previous analysis, the distance from $v$ to $\Z^{\lfloor n/2\rfloor-1}$ is at least, \[ r\lP \frac{1}{432}\cdot\frac{\lfloor r \rfloor}{r^2}\cdot\frac{1}{2\pi x}\rP = \frac{1}{432}\cdot\frac{\lfloor r \rfloor}{r}\cdot \frac{1}{2\pi x} \geq \frac{1}{864}\cdot \frac{1}{2\pi x}, \] whenever $\frac{1}{36\cdot 2\pi x}\geq r\geq 1$.
\end{proof}

{
\begin{remark}\label{rem110120191225} The condition $\gcd(n,k)=1$ in Lemma \ref{lem060120201249} can be dropped. If $\gcd(n,k) = m$, recalling the expression \eqref{eqn1907202020}, we use it with $n'=n/m$ and $k'=k/m$. Also, in the proof of Lemma \ref{lem060120201249} the parameter $r$ is not necessarily an integer number.
 \end{remark}
}

Now, as $v\in\R^{\lfloor n/2\rfloor-1}$ has entries $v_j=\cos\lP \frac{2\pi k}{n} j\rP$ for all $j$, we have $\sqrt{\nicefrac{n}{8}}\leq\| v\|\leq \sqrt{\nicefrac{n}{2}}$. Using Lemma \ref{lem060120201249}, we can estimate the LCD of $v$. Assume that $\theta \leq \frac{1}{72\pi x}$ with $x=\frac{\gcd(n,k)}{n}$. If $\gcd(n,k)\leq n^{1/3}$, by Lemma \ref{lem060120201249}, Remark \ref{rem110120191225} and the definition of LCD for $v$ we get,
\[
 \frac{1}{1728\pi} n^{2/3}\leq \frac{1}{1728\pi x} \leq \dist\lP \theta v, Z^{\lfloor n/2\rfloor - 1}\rP \leq  L \sqrt{\log_{+} \frac{\| \theta v\|}{L}} \leq  L \sqrt{\log_{+} \lP \frac{1}{L}n^{\nicefrac{3}{2}}\rP},
\] which is a contradiction for all large $n$ if we assume $L$ is fixed. We can conclude that LCD of $v$ should satisfy
\begin{equation}\label{eqn120120201327}
D(v) \geq \frac{1}{72\pi} n^{2/3}.
\end{equation} Thus, by definition of LCD and expression \eqref{eqn120120201327} we have, 
\begin{equation}\label{eqn120120201348}
D(n^{0.51}v)\geq n^{-0.51}D(v)\geq \frac{1}{72\pi}n^{2/3-0.51}.
\end{equation}
{
Here, we used $\theta\geq 1$. By Proposition \ref{pro23042020} we can assume $\theta\geq 1/2$. To handle $1>\theta\geq 1/2$, note that we need to give a lower bound of $D(n^{0.51}v)$. So, we can apply Lemma \ref{lem060120201249} to $r=n^{0.51}\theta$, which is greater than $1$ for all sufficiently large $n$, whenever $\frac{n^{2/3}}{72\pi}\geq n^{0.51}\theta$, but this holds since $1>\theta\geq 1/2$. Hence, if $1>\theta\geq 1/2$, then $\dist\lP \theta v, \Z^{\lfloor n/2\rfloor}\rP \geq C n^{2/3}$ for some positive constant $C$.}\\

Using Corollary \ref{thm060120201236} and expression \eqref{eqn120120201348}, we give an upper bound for the second sum of \eqref{eqn29042020}. Thus,
\begin{eqnarray}\label{eqn120120201408}
\sum_{k=1}^{\lfloor n/2\rfloor} \mathcal{L}\lP n^{0.51}S_{n,k}, \varepsilon \rP &  = & \sum_{\gcd(n,k)\leq n^{1/3}} \mathcal{L}\lP n^{0.51}S_{n,k}, \varepsilon \rP + \sum_{\gcd(n,k)> n^{1/3}} \mathcal{L}\lP n^{0.51}S_{n,k}, \varepsilon \rP \nonumber \\
& \leq & \frac{n}{2} \lC\frac{CL}{n^{1.1}}\lP \varepsilon + \frac{72\pi}{n^{2/3-0.51}}\rP\rC + \sum_{\gcd(n,k)> n^{1/3}} \mathcal{L}\lP n^{0.51}S_{n,k}, \varepsilon \rP \nonumber \\
& \leq & C_1\lP\frac{\varepsilon}{n^{0.1}} + \frac{1}{n^{2/3-0.4}}\rP \;\;+ \sum_{\gcd(n,k)> n^{1/3}} \mathcal{L}\lP n^{0.51}S_{n,k}, \varepsilon \rP.
\end{eqnarray}
By Lemma \ref{lem270120191754}, the second term of the sum \eqref{eqn120120201408} has at most $2n^{2/3+\mbox{o}(1)}$ terms. If $n^{1/3}\leq \gcd(k,n)\leq n^{2/3}$, we have, \[ D(n^{0.51}v)\geq n^{-0.51}D(v)\geq \frac{1}{72\pi}n^{1/3-0.51}. \] From Proposition \ref{pro23042020} we have $D(v)\geq 1/2$.  By Lemma \ref{lem270120191754}, the number of positive integers $k$ such that $\gcd(k,n) > n ^{2/3}$ is at most $2n^{1/3+\mbox{o}(1)}$. Then,
\begin{eqnarray} \label{eqn120120201530}
 \sum_{\gcd(n,k)> n^{1/3}} \mathcal{L}\lP n^{0.51}S_{n,k}, \varepsilon \rP  & = & \sum_{n^{2/3}\geq\gcd(n,k)> n^{1/3}} \mathcal{L}\lP n^{0.51}S_{n,k}, \varepsilon \rP \;+ \sum_{\gcd(n,k)> n^{2/3}} \mathcal{L}\lP n^{0.51}S_{n,k}, \varepsilon \rP \nonumber \\
& \leq & 2n^{2/3+\mbox{o}(1)}\lC \frac{CL}{n^{1.1}} \lP \varepsilon + \frac{72\pi}{n^{1/3-0.51}}\rP \rC + \sum_{\gcd(n,k)> n^{2/3}} \mathcal{L}\lP n^{0.51}S_{n,k}, \varepsilon \rP \nonumber \\
& \leq & C_2\lP \frac{\varepsilon}{n^{13/30-\mbox{o}(1)}} + \frac{1}{n^{77/300-\mbox{o}(1)}}\rP + 2n^{1/3+\mbox{o}(1)}\lC \frac{CL}{n^{1.1}} \lP \varepsilon + 2n^{0.51}\rP \rC \nonumber \\
& \leq & C_2\lP \frac{\varepsilon}{n^{13/30-\mbox{o}(1)}} + \frac{1}{n^{77/300-\mbox{o}(1)}}\rP + C_3\lP \frac{\varepsilon}{n^{23/30-\mbox{o}(1)}} + \frac{1}{n^{77/300-\mbox{o}(1)}}\rP \nonumber \\
& \leq & C_4 \lP \frac{\varepsilon}{n^{13/30-\mbox{o}(1)}} + \frac{1}{n^{77/300-\mbox{o}(1)}}\rP.
\end{eqnarray} 

Finally, let $w=(1,\ldots,1)\in\R^{\lfloor n/2 \rfloor - 1}$. Note $\|w\| =n^{0.5}$ and $D(w)\geq \nicefrac{1}{2}$ (Proposition \ref{pro23042020}). By Corollary \ref{thm060120201236}, we have
\begin{equation}\label{eqn200120201742}
\mathcal{L}\lP n^{0.51} S_0,Ê\varepsilon\rP \leq C_5Ê\lP \frac{\varepsilon}{n^{1.01}} + \frac{1}{n^{0.5}}\rP.
\end{equation}

Combining estimates \eqref{eqn120120201408}--\eqref{eqn200120201742}, we get for all large $n$,
\[
\Prob{\min_{0\leq k\leq \lfloor n/2\rfloor} \abs{\lambda^{sym}_k} \leq \varepsilon n^{-0.51}}  \leq C_6\lP \frac{\varepsilon}{n^{0.1}} + \frac{1}{n^{77/300-\mbox{o}(1)}}\rP.
\]\hfill$\Box$

\section{Condition number of a random circulant matrix}\label{sec:conNumToe}

In this section, we give the proof of the first part of Theorem \ref{thm10042020_1}. The proofs of the second part of Theorem \ref{thm10042020_1} and Theorem \ref{thm10042020_2} are similar and they are omitted.\\

\noindent{\bf Proof of Theorem \ref{thm10042020_1}.} By Theorem \ref{lem06082016m02} and Theorem \ref{thm01042018}, for any $\varepsilon>0$ and $\rho\in(0,1/4)$, and all large $n$, we have,
\begin{eqnarray*}
\Prob{\kappa\lP\mathcal{C}_n\rP \leq \frac{C_0}{\varepsilon} n^{\rho+1/2} \lP\log n\rP^{1/2}} & \geq & \Prob{\sigma_{\max}(\mathcal{C}_n) \leq C_0\lP n\log n\rP^{1/2}, \sigma_{\min}^{-1}(\mathcal{C}_n) \leq \frac{1}{\varepsilon} n^{\rho}} \\
& \geq & 1 - \Prob{\sigma_{\max}(\mathcal{C}_n) \geq C_0\lP n\log n\rP^{1/2}} - \Prob{\varepsilon n^{-\rho} \geq \sigma_{\min}(\mathcal{C}_n)} \\
& \geq & 1 - \frac{C_1}{n^2} - C\lP \frac{\varepsilon^2 + \varepsilon}{n^{2\rho}} + \frac{1}{n^{1/2-\textnormal{o}(1)}}\rP \\
& \geq & 1 - C_2\lP \frac{\varepsilon^2+\varepsilon}{n^{2\rho}} + \frac{1}{n^{1/2-\textnormal{o}(1)}}\rP,
\end{eqnarray*} for a positive constant $C_2$ depending on the distribution of $\xi_0$.\hfill $\Box$

\appendix

\section{Salem--Zygmund inequalities} \label{app10042020}

Here, we include the details of the proofs of Theorem \ref{lem06082016m02} and Theorem \ref{thm:LargestSingV}.\\

We introduce two important auxiliary lemmas. The proof of these lemmas can be found in \cite{Barrera}, see Lemma 1.1 and Claim 1 therein. The first lemma is related to a random variable with m.g.f. that sometimes is called {\it locally sub--Gaussian} r.v. The second lemma establishes the existence of a random interval where the function of $W_n(x)$ reaches at least half of its maximum modulus.  Since the second lemma has important aspects of the proof of Theorem \ref{thm:LargestSingV}, we decide to include it here.

\begin{lemma}[Locally sub--Gaussian r.v.]\label{lema1559}
Let $\xi$ be a random variable such that its m.g.f. $M_\xi$ exists in an interval around zero. Assume that $\E{\xi}=0$ and $\E{\xi^2}=s^2>0$. Then there is a $\delta>0$ such that,
\[
M_\xi(t) \leq e^{\nicefrac{\alpha^2 t^2}{2}} \quad \textrm{ for any  } t\in (-\delta,\delta) 
\mbox{ and } \alpha^2 >s^2.
\]
\end{lemma}

\begin{lemma}\label{claim}
There exists a random interval $I\subset \mathbb{T}$ (Lebesgue measure)  of length $\tfrac{8}{3n}$ such that
\begin{align*}
|W_n(x)|\geq \frac{1}{2}
\|W_n\|_\infty \quad \textrm{ for any } x\in I.
\end{align*} 
\end{lemma}

\begin{proof}
In fact, let $p_n(x):=\sum_{j=0}^{n-1} b_j e^{ijx}$ be a trigonometric polynomial on $\mathbb{T}$, where $b_j$ is real number for all $j=0,\ldots,n-1$.
For $x\in \mathbb{T}$, write 
\[g_n(x):=|p_n(x)|^2=\left(\sum_{j=0}^{n-1} b_j \cos(jx) \right)^2 + \left(\sum_{j=0}^{n-1} b_j \sin(jx) \right)^2 
\]
and
\[
h_n(x):= \left( \sum_{j=0}^{n-1} jb_j \cos(jx) \right)^2 + \left(\sum_{j=0}^{n-1} jb_j \sin(jx) \right)^2.
\]
Then
\[
\|p_n\|^2_\infty= \sup_{x\in\mathbb{T}} g_n(x)=\|g_n\|_\infty
\quad \textrm{ and } \quad
\|p^{\prime}_n\|^2_\infty=\sup_{x\in\mathbb{T}} h_n(x).
\]

Recall the Bernstein inequality  $\|p^{\prime}_n\|_\infty\leq n\|p_n\|_\infty$ (see for instance Theorem $14.1.1$, Chapter $14$, page $508$ in \cite{RahSch2002}).
For any $x\in \mathbb{T}$ we have,
\begin{equation}\label{paulo}
\labs g^{\prime}_n(x) \rabs  \leq  4 \|p_n\|_\infty \|p^{\prime}_n\|_\infty  \leq 4n \|p_n\|^2_\infty = 4n \|g_n\|_\infty.
\end{equation}
Since $g$ is continuous there exists $x_0\in \mathbb{T}$ such that 
$g(x_0)=\|g_n\|_\infty$. Moreover, from the Mean Value Theorem and relation \eqref{paulo} we get,
\[
\labs g(x) - g(x_0)\rabs \leq 
\|g^\prime_n\|_\infty \labs x - x_0\rabs \leq 
4n \|g_n\|_\infty \labs x - x_0\rabs,
\]
for any $x\in \mathbb{T}$.
Take  $I:=[x_0 - \frac{3}{16n},x_0+\frac{3}{16n}]\subset 
\mathbb{T}$. Notice that the length of $I$ is $\frac{3}{8 n}$. Moreover, 
\[
\labs g(x) - g(x_0)\rabs \leq \frac{3}{4}  \|g_n\|_\infty,
\quad \textrm{ for any } x\in I. 
\] 
Since $g(x_0)=\|g_n\|_\infty$, from the triangle inequality we deduce
$\frac{1}{4} \|g_n\|_\infty \leq  \labs g_n(x)\rabs$, for any $x\in I$. 
Therefore, 
\[
\frac{1}{2} \|p_n\|_\infty \leq  \labs p_n(x)\rabs,
\;\; \textrm{ for any } x\in I.
\]
\end{proof}

\noindent{\bf Proof of Theorem \ref{lem06082016m02}.}
 By Lemma \ref{lema1559}, there exists a $\delta >0$ such that
\[
M_\xi(t) \leq e^{\alpha^2 t^2/2} \quad \mbox{ for any  } t\in (-\delta,\delta),
\mbox{ where  } \alpha^2 >\E{\xi_0^2}>0.
\]
At first, we assume $W_n(x)=\sum_{j=0}^{n-1} \xi_j e^{ijx}$ is real. Later, we take the imaginary part, but the analysis will be the same. Note that,
\begin{align*}
e^{\nicefrac{\alpha^2 t^2 n}{2}}&=
\prod_{j=0}^{n-1} e^{\nicefrac{\alpha^2 t^2}{2}}
\geq \prod_{j=0}^{n-1} \E{ e^{t \xi_j \cos(jx)}} = \E{\prod_{j=0}^{n-1} e^{t \xi_j \cos(jx)}}=\E{e^{t W_n(x)}}
\end{align*}
for every $t\in (-\delta,\delta)$.

From Lemma \ref{claim}, there exists a random interval $I\subset\mathbb{T}$ of length $\frac{3}{8n}$, {such that $W_n(x)\geq \|W_n\|_\infty$ or $-W_n(x)\geq \|W_n\|_\infty$ on $I$}. Denote by $\mu$ the normalized Lebesgue measure on $\mathbb{T}$. Note that,
\[
\exp\lP \frac{1}{2}t\|W_n\|_\infty\rP = \frac{1}{\mu(I)} \int_{I} \exp\lP \frac{1}{2}t\|W_n\|_\infty\rP dx \leq \frac{1}{\mu(I)}\int_{\mathbb{T}}\lP e^{t W_n(x)}+ e^{-tW_n(x)}\rP dx.
\] Then, for every $t\in(-\delta,\delta)$ we have,
\begin{eqnarray*}
\E{\exp\lP \frac{1}{2}t\|W_n\|_\infty\rP} & \leq & \frac{8n}{3} \E{\int_I\lP e^{t W_n(x)}+ e^{-tW_n(x)}\rP \mu(dx)} \\
& \leq & \frac{8n}{3} \E{\int_{\mathbb{T}}\lP e^{t W_n(x)}+ e^{-tW_n(x)}\rP \mu(dx)} \\
& \leq & \frac{16 n}{3} \exp\lP 3\alpha^2 t^2 n/2\rP.
\end{eqnarray*} From the previous expression we obtain,
\[
\E{\exp\lL \frac{t}{2}\lP \|W_n\|_\infty - 3\alpha^2 t n - \frac{2}{t}\log\lP \frac{16 n}{3}l\rP \rP\rL} \leq \frac{1}{l}
\] for all $l>0$ and every $t\in(-\delta,\delta)$. Taking $l_n = cn^2$ with $c=\nicefrac{3}{16}$, {the inequality $\abs{\frac{\log\lP \frac{16}{3}n\cdot cn^2\rP}{\alpha^2 n}} < \delta^2$ holds for all large $n$}. Let $t_n= \lP\frac{\log\lP \frac{16}{3}n\cdot cn^2\rP}{\alpha^2 n} \rP^{1/2}$. Thus, 
\[
\Prob{\|W_n\|_\infty \geq 5\lP \alpha^2 n \log\lP\frac{16}{3} n \cdot l_n\rP \rP^{1/2}} \leq \frac{1}{l_n} \;\;\; \mbox{ for all large $n$.}
\]
As $e^{ijx} = \cos(jx)+i\sin(jx)$, for all large $n$ {we have},
\[
\Prob{ \| \mbox{Re}(W_n) \|_\infty \geq 5\lP \alpha^2 n \log\lP\frac{16}{3} cn^3\rP \rP^{1/2} }  \leq  \frac{1}{c n^2}
\]
and
\[
\Prob{ \| \mbox{Im}(W_n) \|_\infty \geq 5\lP \alpha^2 n \log\lP\frac{16}{3} cn^3\rP \rP^{1/2} } \leq  \frac{1}{c n^2}.
\] From the above, for some suitable positive constants $C_0,C_1$ depending on the distribution of $\xi_0$, we have that,
\[
\Prob{ \|W_n\|_\infty \geq C_0 \left( n \log n\right)^{\nicefrac{1}{2}} }\leq \frac{C_1}{n^2}, \quad \textrm{ for all large } n.
\] \hfill $\Box$\\

The proof of Theorem \ref{thm:LargestSingV} is similar to the proof of Theorem \ref{lem06082016m02}, but we need to handle the condition $\xi_j=\xi_{n-j}$.\\ 

\noindent{\bf Proof of Theorem \ref{thm:LargestSingV}.} 
Note that Lemma \ref{claim} holds for $W_n^{sym}$. The arguments are the same as in the proof of Theorem \ref{lem06082016m02} up to the analysis of m.g.f. of $W^{sym}_n(x)$, where it is necessary to observe the following. We take the real part of $W^{sym}_n(x)=\sum_{j=0}^{n-1} \xi_j e^{ijx}$. Assume that $n$ is even, the other case is similar. Then,
\begin{eqnarray*}
e^{3\alpha^2 t^2 n/2} \geq  e^{\alpha^2 t^2/2}\prod_{j=1}^{n/2-1} e^{2\alpha^2 t^2} & \geq & \E{e^{t\xi_0}} \prod_{j=1}^{n/2-1} \E{\exp\lP t \xi_j\cos(jx) + t\xi_{n-j}\cos((n-j)x)\rP} \\
& = & \E{e^{t\xi_0} \prod_{j=1}^{n-1} e^{t\xi_j\cos(jx)}} = \E{e^{\mbox{Re}\lP W^{sym}_n(x)\rP}},
\end{eqnarray*} for every $2t\in(-\delta,\delta)$. The next arguments are the same as in the proof of Theorem \ref{lem06082016m02} with the only difference of taking $2t\in(-\delta,\delta)$. Then for some suitable  positive constants $C_0,C_1$ depending on the distribution of $\xi_0$ we have,
\[
\Prob{ \|W^{sym}_n\|_\infty \geq C_0 \left( n \log n\right)^{\nicefrac{1}{2}} }\leq \frac{C_1}{n^2}, \;\textrm{ for all large $n$}.
\] \hfill $\Box$

\section{Arithmetics properties} \label{app30042020}

\subsection{Proof of Lemma \ref{lem270120191754}} 

Denote by $T$ the Euler totient function. Observe that, 
\[
\sum_{\substack{
        k:\;
        \gcd(k,M) \geq  y \\
        1\leq  k \leq  M
    } 
}   1 
\leq 
\sum_{\substack{
        d=\lfloor y\rfloor\\
        d\left| M\right.
    } 
}^M   T\lP\frac{M}{d}\rP. 
\]
Recall that  $T\lP s\rP\leq s - \sqrt{s}$ for all $s\in\N$. Moreover, if $d(s)$ denotes the number of divisors of $s$, then by Theorem 13.12 in \cite{APO},  there exists a positive  constant $C$ such that $d(s) \leq s^{C\lP\log\log\lP s\rP\rP^{-1}}$.
Hence
\begin{eqnarray} \label{eqn220220191514}
\sum_{\substack{
        k:\;
        \gcd(k,M) \geq  y \\
        0 \leq  k \leq M
    } 
}   1 
 \leq 
\lP\frac{M}{\lfloor y\rfloor} - \sqrt{\frac{M}{\lfloor y\rfloor}}\rP M^{C\lP\log\log\lP M\rP\rP^{-1}}  \nonumber 
 \leq  \frac{1}{\lfloor y\rfloor} M^{1+C\lP \log \log M\rP^{-1}}.
\end{eqnarray} \hfill $\Box$

\subsection{Proof of Lemma \ref{lem050320191745}}

We define the following sequence \[P=\lL\exp\lP i\lP j2\pi x - \theta\rP\rP: j=0,\ldots,m-1\rL,\] where $i$ is the imaginary unit. Note that $P$ is a set of points on the unit circle which can be seen as vertices of a regular polygon with $m$ sides inscribed in the unit circle. Since the arguments of the points of the form $\exp\lP i\lP j2\pi x - \theta\rP\rP$ are separated exactly by a distance $2\pi x$, the number of points $\exp\lP i\lP j2\pi x - \theta\rP\rP$ which are in an arc of length $l$ on the unit circle is at least $\frac{l}{2\pi x}-2$.\\

Let $\lC y, y + 3(2\pi x)\rC$ be a subinterval of $[-1,1]$ and we consider the arc $\stackrel{\frown}{A}$ on the unit circle whose projection on the horizontal axis is $\lC y, y + 3(2\pi x)\rC$. If the length of the arc $\stackrel{\frown}{A}$ is $l$, then the number of values $\cos\lP j2\pi x - \theta\rP$ which are still in $\lP y, y + 3(2\pi x)\rP$ is at least $\frac{l}{2\pi x}-2\geq \frac{3\lP 2\pi x\rP}{2\pi x} - 2 = 1$ since $l\geq 3\lP 2\pi x\rP$.\\

Let $s\in\lC-(r-1),(r-1)\rC\cap\Z$. Note that there exists at least one value \[\cos\lP j2\pi x - \theta\rP\in\lP \frac{s}{r} + 3\lP k-1\rP\lP 2\pi x\rP ,\frac{s}{r} + 3k\lP 2\pi x\rP\rP \subset \lC\frac{s}{r},\frac{s+1}{r}\rC\] for all positive integers $k\leq \frac{1}{3 r \lP 2\pi x\rP}$.\\
 
Now, we consider all the values $\cos\lP j2\pi x - \theta\rP\in\lC\frac{s}{r},\frac{s+1}{r}\rC$ and define \[d_j:=\min\lL \abs{\cos\lP j2\pi x - \theta\rP - \frac{s}{r}}, \abs{\cos\lP j2\pi x - \theta\rP - \frac{s+1}{r}}\rL.\] Let $L$ be the biggest integer which satisfies $\lP3\cdot 2 \pi x\rP L \leq \frac{1}{2r}$, or equivalently, $L=\left\lfloor \frac{1}{2r\lP3\cdot 2 \pi x\rP} \right\rfloor$. Therefore, the sum of $d_j$ for all $\cos\lP j2\pi x - \theta\rP \in \lC\frac{s}{r},\frac{s+1}{r}\rC$ is at least
\begin{eqnarray*}
	\sum_{\lambda=1}^L 2\lambda \lP3\cdot 2\pi x\rP  & = & 6\lP 2\pi x\rP \sum_{\lambda=1}^L \lambda \geq  6 \lP 2\pi x \rP \frac{L^2}{2} \\
	& \geq & 3\lP 2\pi x \rP \lP \frac{1}{2} \cdot \frac{1}{\lP2r\rP \lP3\cdot 2\pi x\rP} \rP^2 = \frac{1}{12} \cdot \frac{1}{\lP2r\rP^2 \lP2\pi x\rP},
\end{eqnarray*} where the following inequality was used:
\[
L \geq \frac{1}{2r\lP3\cdot 2\pi x\rP} - 1 \geq \frac{1}{2} \cdot \frac{1}{2r\lP3\cdot 2\pi x\rP}, 
\] which holds if $\frac{1}{2r\lP2\pi x\rP} \geq 6$.
Let $\sigma_s$ be the sum of $d_j$ for each interval $\lC \frac{s}{r},\frac{s+1}{r}\rC$, $s=-(r-1),\ldots,(r-1)$. Since $r\geq 2$, we have,
\begin{equation}\label{eqn280221091800}
\sum_{s=-(r-1)}^{r-1} \sigma_s \geq \lP 2r - 2\rP \lP \frac{1}{12}\cdot \frac{1}{\lP2r \rP^2 \lP 2\pi x\rP}\rP \geq \frac{1}{24}\cdot\frac{1}{\lP 2r\rP\lP 2\pi x\rP}.
\end{equation} From the previous analysis, the distance between the vector $\mathcal{V}\in\R^m$ whose entries are $\mathcal{V}_j = r\cos\lP j2\pi x - \theta\rP$ for $j=0,\ldots,m-1$ with $x=\nicefrac{1}{m}$ to $\Z^m$ is at least
\begin{equation*}\label{eqn280120191810}
r \lP\frac{1}{12}\cdot\frac{1}{\lP 2r \rP\lP2\pi x\rP}\rP = \frac{1}{48}\cdot\frac{1}{2\pi x}, 
\end{equation*} verifying that $\frac{1}{2r\lP2\pi x\rP} \geq 6$ is fulfilled. \hfill$\Box$


\markboth{}{References}
\bibliographystyle{amsplain}

\begin{thebibliography}{99}
\bibitem{Ada} Adamczak, R., On the operator norm of random rectangular Toeplitz matrices, {\it High Dimensional Probability VI}, BirkhŠuser, Basel, $247$--$260$, $2013$.

\bibitem{adhikari2017fluctuations} Adhikari, K., \& Saha, K., Fluctuations of eigenvalues of pattered random matrices, {\it Journal of Mathematical Physics}, $58$, $6$, $2017$, $063301$. 

\bibitem{APO} Apostol, T., {\it Introduction to Analytic Number Theory}, {Undergraduate Texts in Mathematics, Springer--Verlag}, $1976$.

\bibitem{Barrera} Barrera, G., \& Manrique, P., Salem--Zygmund Inequality for Locally sub--Gaussian Random Variables and Random Trigonometric Polynomials, $2016$. arXiv preprint arXiv:1610.05589.

\bibitem{Bha97} Bhatia, R., {\it Matrix Analysis}, Springer--Verlag, 1997.

\bibitem{BorCha2012} Bordenave, C., \& Chafa\"{i}, D., Around the Circular Law, {\it Probability Surveys}, $9$, $2012$, $1$--$89$.

\bibitem{arup2018} Bose, A., \& Saha, K., {\it Random Circulant Matrices}, CRC Press, $2018$.

\bibitem{aruSen2007} Bose, A., \& Sen, A., Spectral norm of random large dimensional noncentral Toeplitz and Hankel matrices, {\it Electronic Communications in Probability}, $12$, $21$--$27$, $2007$.

\bibitem{AruRajKou2009} Bose, A., Subhra, R., \& Saha, K.,  Spectral Norm of Circulant--Type Matrices,  {\it Journal of Theoretical Probability} $24$, $2$, $2011$, $479$--$516$.

\bibitem{burgisser2013condition} B\"{u}rgisser, P., \& Cucker, F., {\it Condition: The geometry of numerical algorithms}, Springer Science \& Business Media, Vol. 349, 2013.

\bibitem{cucker2016probabilistic} Cucker, F., Probabilistic analyses of condition numbers, {\it Acta Numerica}, $25$, $2016$, $321$--$382$.

\bibitem{demidenko2017applications} Demidenko, E., Applications of Symmetric Circulants Matrices to Isotropic Markov Chain Models and Electrical Impedance Tomography, {\it Advances in Pure Mathematics}, $7$, $2$, $2017$, $188$--$198$.

\bibitem{demmel1987geometry} Demmel, J., The geometry of Ill--conditioning, {\it Journal of Complexity}, $3$, $2$, $1987$, $201$--$229$.

\bibitem{David2012} Davis, P., {\it Circulant matrices}, American Mathematical Soc., $2013$.

\bibitem{erdos} Erd\"os, P., Problems and Results on Polynomials and Interpolation, {\it Aspects on Contemporary Complex Analysis}, $1980$, $383$--$391$.

\bibitem{gray2006toeplitz} Gray, R., {\it Toeplitz and circulant matrices: A review}, Foundations and Trends in Communications and Information Theory, $2006$.

\bibitem{golub2012matrix} Golub, G., \& Van Loan, C., {\it Matrix computations}, JHU press, $2012$.

\bibitem{govaerts1989singular} Govaerts, W., \& Pryce, J.D., A singular value inequality for block matrices, {\it Linear algebra and its applications}, $125$, $141$--$148$, $1989$.

\bibitem{Kah1985} Kahane, J.,  {\it Some Random Series of Functions}, Second Edition, Cambridge University Press Cambridge, $1985$.

\bibitem{Lit} Litvak, A. E., Tikhomirov, K., \& Tomczak-Jaegermann, N., Small ball probability for the condition number of random matrices, {\it Geometric Aspects of Functional Analysis}, Springer, Cham, $125$--$137$, $2020$.

\bibitem{Liv} Livshyts, V., Tikhomirov, K., \& Vershynin, R.,  The smallest singular value of inhomogeneous square random matrices. {\it arXiv preprint arXiv:1909.04219}, 2019.

\bibitem{Luh2018} Luh, K., \& Vu, V. Sparse random matrices have simple spectrum.  To appear {\it Annales de l'Institut Henri Poincar\'e Probabilit\'es et Statistiques}, {\it arXiv preprint arXiv:1802.03662}, 2018.

\bibitem{Mec2007} Meckes, M., On the spectral norm of a random Toeplitz matrix, {\it Electronic Communications in Probability}, $2007$, $12$, $315$--$325$.

\bibitem{Mec2009} Meckes, M., Some Results on Random Circulant Matrices, {\it High Dimensional Probability: The Luminy $V$}, Institute of Mathematical Statistics, $2009$, $213$--$223$.

\bibitem{ng2004iterative} Ng, M. K., {\it Iterative Methods for Toeplitz Systems}, Oxford University Press, $2004$.

\bibitem{RahSch2002} Rahman, Q., \& Schmeisser, G.,  {\it Analytic Theory of Polynomials: Critical Points, Zeros and Extremal Properties}, {Oxford Science Publications}, $2002$.

\bibitem{RV1} Rudelson, M., \& Vershynin, R., The Littlewood--Offord Problem and Invertibility of Random Matrices, \textit{Advances in Mathematics}, $218$, $2008$, $600$--$633$.

\bibitem{RV} Rudelson, M., \& Vershynin, R., Non-Asymptotic Theory of Random Matrices: Extreme Singular Values, {\it Proceedings of the International Congress of Mathematicians Hyderabad Volumen III, India, Editor: Rajendra Bhatia}, $2010$, $1576$--$1602$.

\bibitem{RV2} Rudelson, M., \& Vershynin, R., No--gaps Delocalization for General Random Matrices, {\it Geometric and Functional Analysis}, $26$, $6$,  $2016$, $1716$--$1776$.

\bibitem{SenVir2013} Sen, A., \& Vir\'ag, B., The Top Eigenvalue of the Random Toeplitz Matrix and the Sine Kernel, {\it Annals of Probability}, $41$, $6$,  $2013$, $4050$--$4079$.

\bibitem{vershynin2014invertibility} Vershynin, R., Invertibility of symmetric random matrices, {\it Random Structures \& Algoritms}, $44$, $2$, $2014$, $135$--$182$.

\bibitem{pan2001structured} Pan, V., {\it Structured Matrices and Polynomials: Unified Superfast Algorithms}, Birkh{\"a}user Boston, $2001$.

\bibitem{panSZ2015} Pan, V., Svadlenja, J., \& Zhao, L., Estimating the norms of random circulant and toeplitz matrices and their inverses, {\it Linear algebra and its applications}, $468$, $2015$, $197$--$210$.

\bibitem{Weber2006} Weber, M., On a Stronger Form of Salem--Zygmund's Inequality for Random Trigonometric Sums with Examples, {\it Periodica Mathematica Hungarica}, $52$, $2$,  $2006$, $73$--$104$.

\bibitem{xiang2007structured} Xiang, H., \& Wei, Y., Structured mixed and componentwise condition numbers of some structured matrices, {\it Journal of computational and applied mathematics}, $202$, $2$, $2007$, $217$--$229$.
\end{thebibliography}

\end{document}